\documentclass[11pt,a4paper]{amsart}
\usepackage{hyperref}
\usepackage{graphics,epsfig,color}
\usepackage{amsmath,amsthm,amssymb}
\usepackage{framed}
\usepackage{dsfont}

\theoremstyle{plain}
\newtheorem{theo}{Theorem}[section]

\newtheorem{lem}[theo]{Lemma}

\newtheorem{prop}[theo]{Proposition}
\theoremstyle{remark}
\newtheorem{rmk}[theo]{Remark}

\numberwithin{equation}{section}

\newcommand{\R}{\mathbb{R}}
\newcommand{\N}{\mathbb{N}}
\renewcommand{\leq}{\leqslant}
\renewcommand{\geq}{\geqslant}

\newcommand{\umu}{u^\mu}
\newcommand{\del}{\partial}
\newcommand{\Dt}{\dfrac{d}{dt}}
\DeclareMathSymbol{\perp}{\mathrel}{symbols}{"3F}
\DeclareMathOperator{\dist}{dist}

\DeclareMathOperator{\curl}{curl}
\DeclareMathOperator{\supp}{supp}
\DeclareMathOperator{\divv}{div}
\DeclareMathOperator{\D}{D}

\newcommand{\intO}{\int_{\Omega}}
\newcommand{\intBO}{\int_{\del\Omega}}

\def\XXint#1#2#3{{\setbox0=\hbox{$#1{#2#3}{%
\int}$ }
\vcenter{\hbox{$#2#3$ }}\kern-.6\wd0}}

\allowdisplaybreaks

\title[Degenerate lake equations]{Degenerate lake equations: classical solutions and vanishing viscosity limit}
\author[B. Al Taki and C. Lacave]{Bilal Al Taki and Christophe Lacave}
\date{\today}

\address[B. Al Taki]{Sorbonne Université, CNRS, Laboratoire Jacques-Louis Lions, 75005 Paris, France.}
\email{altaki@ljll.math.upmc.fr}

\address[C. Lacave]{Institut Fourier, UMR 5582, CNRS, Université Grenoble Alpes, 38058 Grenoble cedex 9, France.}
\email{Christophe.Lacave@univ-grenoble-alpes.fr}

\begin{document}

\maketitle

\begin{abstract}
The objective of this paper is twofold. First, we show the existence of global classical solutions to the degenerate inviscid lake equations. This result is achieved after revising the elliptic regularity for a degenerate equation on the associated stream-function, and adapting the method used for construction of classical solutions to the incompressible Euler equations. Second, we show that the weak solutions of the viscous lake equations converge to classical solutions of the inviscid lake equations when the viscosity coefficient goes to zero, which constitutes an important physical validation of these models. The later result is achieved by the use of energy method as in the proofs of Kato-type theorems. This method also allows us to expose a convergence rate.
\end{abstract}
\maketitle


\section{Introduction}

A central problem in the mathematical analysis of fluid dynamics is the asymptotic limit of the fluid flow as viscosity goes to zero. In the presence of physical boundaries with the usual no-slip condition, this problem is essentially open even for the incompressible Navier-Stokes equations due to the possible appearance of boundary layers. In this paper, we assume that a slip boundary condition of Navier-type holds. With this condition, the vanishing viscosity problem has been extensively studied when the fluid is described by the incompressible Navier-Stokes equations \cite{MikRob,Lopes2Planas,Kelliher, IftimieSueur}. We propose to study the vanishing viscosity limit of the solution to the degenerate viscous lake equations with the general Navier boundary condition. These equations model an incompressible viscous flow of a fluid in a lake when the horizontal velocity scale is large compared to the depth $b:\overline{\Omega}\rightarrow \R_+$ (shallow water) but small compared to the gravity (small Froude number: ${\rm Fr}\ll1$). This formal characterization has been mathematically justified in \cite{Br-Gi-Li-der-lake} where the authors derived the viscous lake equations from the viscous shallow water equations by letting the Froude number go to zero when the initial height converges to a non-constant function depending on the space variable. The obtained equations are the following
\begin{equation}\label{Viscous-lake}
\begin{cases}
\del_{t}(bu^{\mu})+\divv(bu^{\mu}\otimes u^{\mu})-2\mu\divv(b\D(u^{\mu})+b\divv u^{\mu}\,\mathbb{I})+b\nabla p^{\mu}=0,\\
\divv(bu^{\mu})=0, 
\end{cases}
\end{equation}
for $(t,x)\in (0,T)\times \Omega$ with $\Omega$ a bounded domain in $\R^2$. Here, $u^{\mu}=u^{\mu}(t,x)=(u^{\mu}_{1}(t,x),u^{\mu}_{2}(t,x))$ stands for the two-dimensional horizontal component of the fluid velocity, and $p^\mu=p^\mu(t,x)$ represents the pressure. These two functions are the unknown of the system. We complete System~\eqref{Viscous-lake} with the general Navier boundary condition:
\begin{equation}\label{VS-B-NV}
bu^{\mu}\cdot n=0,\;\;\;2b(\D(u^{\mu})\cdot n+\divv u^{\mu}\,\mathbb{I}\cdot n)\cdot\tau+\eta_{\mu}\, b\, u^{\mu}\cdot\tau=0,\;\;\;{\rm on}\;\; \del \Omega,
\end{equation}
where $n$ is the inward-pointing unit normal vector, $\tau=n^\perp=(-n_{2},n_{1})^T$ the unit tangent vector, and $\eta_\mu$ is a turbulent boundary drag function defined on $\del \Omega,$ that we allow it here to depend on the viscosity coefficient $\mu$, that is it, we assume $0\leq \eta_\mu \leq\eta \mu^{-\beta}$ with $\eta$ being a positive constant and $\beta<1$. These boundary conditions, introduced by {\sc H. Navier}, assume that the tangential slip velocity, rather than being zero, is proportional to the tangential stress. From a physical point of view, these conditions has been justified, for instance, in two-dimensional geophysical models, where the viscosity take into account various turbulent effect at small scale (see \cite{Ped-book}). We emphasize here, that our motivation beyond allowing $\eta_\mu$ depending on $\mu$ comes back from \cite{Wang-xi, padd} where such choice has been taken when dealing with the incompressible Navier-Stokes equations in order to understand the transition between the unstable Dirichlet case and the stable Navier case.
\medskip

The well-posedness question of System~\eqref{Viscous-lake}-\eqref{VS-B-NV} was firstly studied by {\sc D. Levermore} and {\sc M. Sammartino} \cite{Lev-Sa} in the nondegenerate case, namely when $b\geq c>0$. Existence, uniqueness, and regularity of weak solutions were then showed by adapting the classical proof used for the 2D incompressible Navier-Stokes equations. It is worth noting that the weighted Sobolev spaces introduced by the authors are actually equivalent to the standard Sobolev spaces since the non-degeneracy fact of the vertical depth $b$. In the degenerate case, this equivalence is no longer true, and consequently, precise definition of the weights under consideration must be determined. Using {\it Muckenhoupt} class of weights, the first author extended in \cite{Lake-B} the above result to the degenerate case by proving the global-in-time existence of a unique weak solution to System~\eqref{Viscous-lake}-\eqref{VS-B-NV}.

\medskip

Neglecting the viscous term in \eqref{Viscous-lake}, i.e., taking $\mu=0$, System~\eqref{Viscous-lake} reduces to the so-called inviscid lake equations which read as
\begin{equation}\label{Inviscid-Lake}
\left\{
\begin{array}{l}
\del_{t}(bu)+bu\cdot\nabla u+b\nabla q=0,\vspace*{0.2cm}\\
\divv(bu)=0 , \quad (bu)\cdot n=0.
\end{array}
\right.
\end{equation}
In the case where $b$ is constant, System~\eqref{Inviscid-Lake} becomes the well-known two dimensional Euler equations and the well-posedness is widely known due to the work of {\sc W. Wolibner} \cite{Wolibner} and {\sc V. I. Yudovich} \cite{Yud-euler}. As for the 2D-Euler equations, the notion of vorticity plays a prevalent role. Here, we introduce the potential vorticity $\omega$ as
\begin{equation*}
 \omega:=\frac{1}{b}\curl u=\frac{\partial_{1} u_{2}-\partial_{2} u_{1}}{b}
\end{equation*}
which satisfies the continuity equation together with the incompressibility condition
\begin{equation*}
 \partial_t(b\omega)+\divv(b u \omega)=0, \qquad \divv(bu)=0.
\end{equation*}
This amounts to the following vorticity formulation
\begin{equation}\label{Inviscid-vort}
\left\{
\begin{array}{l}
\partial_t\omega+u\cdot\nabla\omega =0, \vspace*{0.2cm}\\
\curl u = b\omega, \quad \divv(bu)=0, \quad (bu)\cdot n=0.
\end{array}
\right.
\end{equation}

When the depth $b$ varies but is bounded away from zero, the well-posedness (existence and uniqueness) of \eqref{Inviscid-Lake} was established by {\sc C. D. Levermore, M. Olivier} and {\sc E. S. Titi} in \cite{Lev-Titi-lake}.
 In \cite{br-me-lake}, {\sc D. Bresch} and {\sc G. M\'{e}tivier} allow the varying depth to vanish on the boundary. The essential tool in establishing the well-posedness in \cite{br-me-lake} is an elliptic regularity for a degenerate equation on the associated stream function. This estimate is highly non trivial to obtain if the depth vanishes, and the proof is related to a careful study of the associated Green function. 
More recently, {\sc C. Lacave, T. Nguyen} and {\sc B. Pausader} \cite{la-pau} extended the work in \cite{br-me-lake} treating the case of singular domains and rough bottoms. They proved that the inviscid lake equations are structurally stable under Hausdorff approximations of the fluid domain and $L^{p}$ perturbations of the depth. This study was extended for an evanescent or emergent island \cite{he-la-miot}.
\medskip

A natural question to ask is whether the solution of the viscous lake equations \eqref{Viscous-lake} converges to the solution of the inviscid lake equations \eqref{Inviscid-Lake} when the viscosity coefficient tends to zero. In this paper, we give an answer to this question. However, in order to obtain such a result, some additional regularity properties must be showed either on the weak solutions of the viscous lake equations, which according to \cite{Lake-B} belong to\footnote{See Section 2 for a definition of the space $H_b$ and $V_b$.} $L^{\infty}(0,T; H_{b})\cap L^{2}(0,T; V_{b})$, or on the solutions of the inviscid lake equations, which according to \cite{br-me-lake, la-pau} belong to $C(0, T; W^{1,p}(\Omega))$, for any $p<\infty$. Unfortunately, the study of regular solutions of the viscous lake equations is quite difficult. This is because of the degeneracy fact of the depth near the boundary, and also the fact that the vorticity equation associated to this model is singular and whence not helpful (see \cite{Lake-B}). In contrast, the inviscid model admits a nice vorticity-stream formulation \eqref{Inviscid-vort}. For this reason, we prove in the first part of this paper the existence of classical solutions of the inviscid model (in a class $C^1$). In the second part of this paper, we show the strong convergence uniformly in time in $L^2_b(\Omega)$ of the unique weak solution $u^\mu$ of \eqref{Viscous-lake} as $\mu$ goes to zero to the classical solution $u$ of \eqref{Inviscid-Lake} provided that the initial data converges in $L^2_b(\Omega)$ to a sufficiently smooth limit.

\section{Main results} 

In this section, we state the main theorems proved in this paper. First, we start by presenting our results on the inviscid model. At this stage, we emphasize that the Yudovich-type solution to the inviscid lake equations \eqref{Inviscid-Lake} is unique (\cite{br-me-lake,la-pau}), and so in this part of the paper, we are only interested in the question of existence of classical solutions. In the second part, and after recalling the notion of weak solutions of the viscous lake model \eqref{Viscous-lake}, we state the second main result of this paper which concerns the vanishing viscosity limit.

\subsection{Classical solutions of the inviscid lake equations}\label{def-invi} 

The first main ingredient of this article is to provide classical solutions to the inviscid lake equations \eqref{Inviscid-Lake}. As said in the introduction, the crucial quantity in these equations is the potential vorticity $\omega$. It is worth mentioning that, the $L^p-$norms of $b^{\frac{1}{p}}\omega$ is a conserved quantity for any $p\in [1,+\infty]$, which provides an important estimate on the solution. When $\Omega$ is simply connected, this leads to an equivalent system in terms of the vorticity \eqref{Inviscid-vort}. When the domain is not simply-connected (because of the presence of islands), we have to define the circulation and the $b$-harmonic functions. This makes the construction of the velocity in terms of the vorticity more complicated. We refer to \cite{la-pau} for such a Hodge decomposition, and we propose here to study the simplest case where the lakes do not have islands.

Let $(\Omega,b)$ satisfying the following conditions
\begin{equation}\label{b-cond}
\begin{split}
& \Omega \mbox{ a simply-connected open and bounded set where } \partial\Omega\in C^3, \\
&b(x)=c(x)\varphi^{\alpha}(x) \quad\mbox{with } \alpha \geq 0 \mbox{ and } c(x)\geq c_{0}>0 \mbox{ on } \Omega,\\
& \Omega=\{\varphi>0\}\quad\mbox{with } c,\varphi\in C^{3}(\overline{\Omega}) \mbox{ and } \nabla \varphi\neq0 \mbox{ on } \del \Omega. 
 \end{split}
\end{equation}
With these conditions, we can consider the simplest case of a non-vanishing shore ($\alpha=0$) or the more realistic case of a vanishing topography ($\alpha>0$). The divergence free condition and the vanishing of the normal component of $bu$ allow us to state that $bu=\nabla^{\perp} \psi$ where $\psi$ is the unique solution belonging to 
\[
X_{b}(\Omega):=\{ \phi\in H^1_{0}(\Omega); \ b^{-1/2}\nabla \phi\in L^2(\Omega)\}
\]
of
\begin{equation}\label{ellip-eq}
\divv\Big(\dfrac{1}{b}\nabla \psi\Big)=f \quad \text{in} \quad \Omega \qquad \psi\vert_{\del \Omega}=0,
\end{equation}
with $f=b\omega\in L^\infty(\Omega)$.

The existence and uniqueness of $\psi\in X_{b}(\Omega)$ is used in \cite{br-me-lake} to study the inviscid lake equations \eqref{Inviscid-Lake}. In that paper, the authors derived higher elliptic estimates (see \eqref{ell-reg} below) for the solution of \eqref{ellip-eq} in order to prove the existence and uniqueness of weak solutions by the {\sc Yudovich}'s argument. We go further in this direction by showing that this solution enjoys moreover the so-called log-Lipschitz estimate (see \eqref{log-lip} below), which constitutes an important ingredient of our proof of existence of classical solutions to System~\eqref{Inviscid-Lake}.
\begin{theo}\label{ellip-est}
Let $(\Omega,b)$ verifying \eqref{b-cond}. There exists $C>0$, and for any $M>0$ and $K\Subset \Omega$, there exists $C_{M,K}>0$ such that the following assertions hold.
\medskip

(i) {\bf Almost Lipschitz regularity}. For a given $f \in L^{\infty}(\Omega)$, the vector field $u=\frac1b\nabla^{\perp}\psi$ where $\psi$ is the solution of \eqref{ellip-eq} is almost Lipschitz, more precisely
\begin{equation}\label{ell-reg}
\| u\|_{L^\infty(\Omega)}+ \frac1p\| \nabla u\|_{L^p(\Omega)}\leq C \|f\|_{L^\infty(\Omega)}, \quad \forall p\in [2,\infty);
\end{equation}
and 
\begin{equation}\label{log-lip}
|u(x)-u(y)| \leq C \|f\|_{L^\infty(\Omega)} |x-y| \Big(1+\big|\ln |x-y| \big|\Big), \quad \forall x,y\in \Omega.
\end{equation}

\medskip

(ii) {\bf $C^1$ regularity}. In addition, if $f \in C^{1}(\overline{\Omega})$ such that $\|f\|_{L^\infty(\Omega)}\leq M$, the vector field $u$ belongs to $C^{1,\beta}(\overline{\Omega})$, for any $\beta<1$. Moreover, we have 
\begin{equation}\label{C1-ln}
\| \nabla u \|_{L^{\infty}(K)}\leq C_{M,K} \ln \Big(2+\|\nabla f \|_{L^\infty(\Omega)}\Big),
\end{equation}
and 
\begin{equation}\label{C1alpha}
\| u \|_{C^{1,\beta}(\overline{\Omega})}\leq C \| f \|_{C^1 (\overline{\Omega})}.
\end{equation}
\end{theo}

Even if \eqref{C1-ln} seems to be weaker than \eqref{C1alpha}, the sublinear estimate of the velocity in terms of $\|\nabla f\|_{L^\infty}$ is crucial for the fixed-point iteration method used to prove the existence of classical solutions.

\begin{theo}\label{Th-WP}
Let $(\Omega,b)$ verifying \eqref{b-cond}. Let $\omega_{0}\in C^1_{c}(\Omega)$. There exists a unique pair $(u,\omega) \in (C^1([0,+\infty) \times \overline{\Omega}))^2$ solution to \eqref{Inviscid-vort} such that $\omega(0,\cdot) = \omega_{0}$. There exists a unique $u \in C^{1}([0,+\infty) \times \overline{\Omega})$ solution to \eqref{Inviscid-Lake} such that $\curl u(0,\cdot) = b\omega_{0}$. Moreover, for any $T>0$ there exists $\delta_{T}>0$ such that 
\[\dist(\supp \omega(t,\cdot),\partial\Omega)=\dist(\supp \curl u(t,\cdot),\partial\Omega)\geq \delta_{T} \quad \forall t\in [0,T].\]
\end{theo}

\subsection{Vanishing viscosity limit for the viscous lake equations}\label{def-viscous}

We define the following spaces
\begin{equation*}
H_{b}=\{u;u\in L^{2}_{b}(\Omega),\divv(bu)=0,\;bu\cdot n=0\; {\rm on}\,\, \del \Omega\},\\
\end{equation*}
\begin{equation*}
V_{b}=\{u;u\in H^{1}_{b}(\Omega),\divv(bu)=0,\;bu\cdot n=0\; {\rm on}\,\, \del \Omega\},
\end{equation*}
where we denoted by $L^{q}_{b}(\Omega)$ as the set of measurable functions $u$ on $\Omega$ such that
$$\| u \|_{L^{q}_{b}(\Omega)}:=\Big(\int_{\Omega}|u(x)|^{q} b(x)\,dx\Big)^{\frac{1}{q}}<\infty\qquad 1<q<\infty.$$

For a given viscosity $\mu>0$ and $u^{\mu}_{0}\in H_{b},$ we say that $u^{\mu}\in L^{\infty}(0,T; H_{b})\cap L^{2}(0,T; V_{b})\cap C([0,T]; H_{b}-weak)$ is a weak solution of \eqref{Viscous-lake}-\eqref{VS-B-NV} with initial velocity $u^{\mu}_{0}$ if the following properties hold:
\begin{itemize}
\item the initial condition $u^\mu_{0}$ holds in a weak sense:
\begin{equation*}
\Big(\int_{\Omega}u^{\mu}\cdot v\,b\,dx\Big)(0,x)=\int_{\Omega}u^{\mu}_{0}\cdot v\,b\,dx \quad \forall \; v\in V_{b};
\end{equation*}
\item the variational formulation holds in $\mathcal{D}^\prime (0,T)$:
 \begin{multline*}
\Dt \intO u^\mu \cdot v \,b\,dx+\intO ((u^\mu\cdot\nabla) u^\mu)\cdot v \,b\,dx
+2\mu\int_{\Omega}\D(u^{\mu}):\D(v)\,b\,dx\\+2\mu\int_{\Omega}\divv u^{\mu} \divv v\,b\,dx+\mu \int_{\del \Omega} \eta_\mu u^{\mu}\cdot v \,b\,ds=0\quad \forall v\in V_b.
\end{multline*}
\end{itemize}
The existence and uniqueness of weak solutions in the sense of the above definition was shown under the following assumptions on the domain ($\Omega,b$)
\begin{equation}\label{b-cond2}
\begin{split}
& \Omega \mbox{ a simply-connected open and bounded set where } \partial\Omega\in C^3, \\
&b(x)=c(x)\varphi^{\alpha}(x) \quad\mbox{with } 0\leq \alpha <\frac12 \mbox{ and } c(x)\geq c_{0}>0 \mbox{ on } \Omega,\\
& \Omega=\{\varphi>0\}\quad\mbox{with } c,\varphi\in C^{3}(\overline{\Omega}) \mbox{ and } \nabla \varphi\neq0 \mbox{ on } \del \Omega. 
 \end{split}
\end{equation}
The non-vanishing shore ($\alpha=0$) is an adaptation of standard proofs \cite{Lev-Sa}. In the vanishing case ($\alpha>0$), this result was obtained by the first author for {\it Muckenhoupt} class of weights. These family of weights was introduced by {\sc B. Muckenhoupt} in his seminal paper \cite{Muck}. It consists of those weights $w(x)$ for which the Hardy-Littlewood maximal operator is bounded on $L^p(dw)$. In \cite{Lake-B}, it was used that the above conditions \eqref{b-cond2} imply that the function $b$ belongs to the class $ \mathcal{A}_{3/2}$ of Muckenhoupt weights. For more details see \cite{Muck, Lake-B}.
Notice that, with our choice of $b$ \eqref{b-cond2}, the boundary integral term in the above weak formulation vanishes, that is it, the space $L^p_b(\del \Omega)$ has sense only if $\alpha=0$, see Sections 9.13 and 9.14 in \cite{Kufner-book}. Recall that the characterization of the trace operator when the weight $b$ behaves like a distance to the boundary was specified in \cite{Nik}. Despite this, we do not suppress the boundary integral term involving $b$ for the sake of possible generality in the future where the weight $b$ could be taken not identically zero at one component of the boundary, for instance on an island.

\medskip

The third main theorem of this article concerns the vanishing viscosity limit, which shows rigorously that the inviscid lake model is a relevant approximation for slightly viscous lakes.
\begin{theo}\label{Th-Vanishing}
Let $(\Omega,b)$ verifying \eqref{b-cond2}, $\eta>0$ and $\beta\in [0,1)$. Let $u^\mu$ be the unique solution of the viscous lake model with initial data $u^\mu_{0} \in L^{2}_b(\Omega)$, and $u$ the unique solution of the inviscid lake equation with initial data $u_{0}$ such that $\curl u_{0} \in C^{1}_c(\Omega)$. If $u_{0}^{\mu}$ converges to $u_{0}$ in $L^{2}_{b}(\Omega)$ as $\mu$ goes to zero and if $0\leq\eta_{\mu}\leq \eta \mu^{-\beta}$, then $u^{\mu}$ converges to $u$ in $L^\infty_{\rm loc}(\R^+ ; L^{2}_{b}(\Omega))$. More precisely, for every $T>0$, there exists $C_{T}>0$ which depends only on $\Omega,b, \eta, \beta,\curl u_{0}$ and $T$ (but independent of $u_{0}^{\mu}$ and $\mu$) such that
\[
\sup_{t\in [0,T]} \|u^{\mu}(t,\cdot)-u(t,\cdot)\|_{L^{2}_{b}(\Omega)}\leq C_{T} \big(\mu^{\frac{1-\beta}2} + \|\umu_{0} -u_{0}\|_{L^2_b(\Omega)} \big).
\]
\end{theo}

In particular, we can consider $u_{0}^{\mu}=u_{0}$ in the above theorem. For $b$ vanishing on $\partial \Omega$, which is the case in \eqref{b-cond2} when $\alpha>0$, we need no assumption on $\eta_{\mu}$ because it should not play any role in the variational formulation, and we could replace $\beta$ by zero in the previous estimate.

\subsection{Plan of the paper}

The rest of this article is dedicated to proving these three theorems. The next section provides the proof of Theorem~\ref{ellip-est}, which constitutes of two subsections. In the first subsection, we show that the estimates presented in Theorem~\ref{ellip-est} hold locally in $\Omega$, while in the last one, we prove that the estimates \eqref{ell-reg}-\eqref{log-lip} and \eqref{C1alpha} hold up to the boundary. The tools used in each subsection are different, and they are also of independent interest. The main point in proving the local estimates is to show the link between the green kernel associated to \eqref{ellip-eq} and the Laplace's fundamental solution. Concerning the proof near the boundary, the main ingredients are the papers \cite{BCM, gou-shima} and \cite{br-me-lake}, which are devoted on elliptic estimates for a class of degenerate equations similar to \eqref{ellip-eq} studied in this paper. We will follow the same lines introduced in \cite{br-me-lake} for our proof of the log-Lipschitz estimate \eqref{log-lip} near the boundary.

Having obtained Theorem~\ref{ellip-est}, Section~\ref{sec-pointfixed} aims to prove the existence of classical $C^1-$solutions of \eqref{Inviscid-Lake}, namely Theorem~\ref{Th-WP}, by an usual fixed-point iteration method. 
In Section~\ref{section-van-vis}, we study the vanishing viscosity limit of the solutions to the viscous lake model, and using energy methods we give a proof of Theorem~\ref{Th-Vanishing}.

\section{Regularity of solutions for degenerated elliptic problems: proof of Theorem~\ref{ellip-est}}

Let $(\Omega,b)$ verifying \eqref{b-cond}, this section is dedicated to the proof of Theorem~\ref{ellip-est}. We omit the non-vanishing topography case ($\alpha=0$) because the estimates in this case follow from \cite{ag-do-ni2, Gi-Tru}\footnote{Except \eqref{C1-ln} which is less classical, but it would be clear that Section~\ref{sec-loc-est} can be applied for $b\geq c>0$.}. For the vanishing case, the degeneracy of the topography $b\vert_{\partial\Omega} = 0$ raises many mathematical difficulties because our equations turn out to be either degenerate \eqref{Viscous-lake}-\eqref{Inviscid-Lake} or singular \eqref{ellip-eq}. When $\alpha>0$, we can choose, without any loss of generality, $c\equiv 1$ in \eqref{b-cond} and \eqref{b-cond2}, or on the contrary, in the neighborhood of the boundary, we can decide to keep $c$ and to replace $\varphi$ by $\dist(x,\del \Omega)$.

When $f\in L^p(\Omega)$ for some $p>1$, the existence and the uniqueness of $\psi\in X_{b}(\Omega)$ for \eqref{ellip-eq} is already known by Lax-Milgram type argument, by using that $b$ is bounded (see, e.g., \cite[Prop. 2.3]{la-pau}). From the energy estimate, we deduce that
\begin{equation}\label{energy}
\| \sqrt b u \|_{L^2(\Omega)} = \Big\| \frac{\nabla\psi}{\sqrt b} \Big\|_{L^2(\Omega)} \leq C \| f \|_{L^\infty (\Omega)}, 
\end{equation}
where $C$ depends only on $\|b\|_{L^\infty(\Omega)}$ and $\Omega$. Let us recall that $C^\infty_{c}(\Omega)$ is dense in $X_{b}(\Omega)$ \cite[Lem. 2.1]{la-pau}.

\subsection{Local elliptic regularity}\label{sec-loc-est} 
In this subsection, we fix $K\Subset \Omega$ and we look for elliptic regularity in $K$. Even if elliptic problems are not local, it is clear that singularities at the boundary will not give singular behavior in $K$ and we already know that local elliptic regularity would give estimates on $\psi$ in $W^{2,p}(K)-$norms. Nevertheless, for \eqref{log-lip} and \eqref{C1-ln}, we need some estimates which are very close to the known estimates in the full plane, and where the precise formula of Laplace's fundamental solution is crucial, in particular for \eqref{C1-ln}. For instance, for the fixed-point procedure used in the proof of existence of classical solutions, it will be important to have a sublinear estimate of the velocity in terms of $\|\omega\|_{C^1}$ similar to estimate \eqref{C1-ln}. For this reason, we are interested by a relation between the Green kernel associated to \eqref{ellip-eq} and the Green kernel in the full plane. Hence, we consider $G_{\Omega,b}$ solution of the following system for all $y\in \Omega$:
\begin{equation}\label{eq.Green}
\left\{\begin{aligned}
& \divv_{x}\Big( \frac1{b(x)} \nabla_{x} G_{\Omega,b}(x,y)\Big)=\delta(x-y)\text{ in }\mathcal{D}'(\Omega), \\
 & G_{\Omega,b}(x,y)=G_{\Omega,b}(y,x)\text{ for all }x\in \Omega, \ G_{\Omega,b}(x,y)=0 \text{ for all }x\in \partial\Omega,
\end{aligned}\right.
\end{equation}
such that the solution of \eqref{ellip-eq} can be written as $\psi(x)=\int_{\Omega} G_{\Omega,b}(x,y) f(y)\, dy$.

The Green kernel in the full plane is $G_{\R^2}(x,y)=\frac1{2\pi}\ln|x-y|$ whereas in the disk becomes
\[
G_{D}(x,y)=\frac1{2\pi} \ln \frac{|x-y|}{|x-y^*||y|},
\]
with the notation $z^*=z/|z|^2$. Such a formula can be adapted to any simply connected bounded set $\Omega$ thanks to a Riemann mapping $\mathcal{T}:\ \Omega\to B(0,1)$:
\[
G_{\Omega}(x,y)=\frac1{2\pi} \ln \frac{|\mathcal{T}(x)-\mathcal{T}(y)|}{|\mathcal{T}(x)-\mathcal{T}(y)^*||\mathcal{T}(y)|}.
\]
In a recent article \cite{DekeyVanSch}, {\sc J. Dekeyser} and {\sc J. Van Schaftingen} have noticed that $(x,y)\mapsto G_{\Omega}(x,y)\sqrt{b(x)b(y)}$ satisfies almost \eqref{eq.Green}, up to a corrector of lower order. Unfortunately, that article considered only the non-vanishing topography case $b(x)\geq c_{0}>0$, so in the sequel we give the details of its generalization to the vanishing topography case.

For all $y\in \Omega$, we define $x\mapsto S_{\Omega,b}(x,y)$ such that 
\begin{equation}\label{eq.GreenSb}
\left\{\begin{aligned}
& \divv_{x}\Big( \frac1{b(x)} \nabla_{x} S_{\Omega,b}(x,y)\Big)= G_{\Omega}(x,y) \sqrt{b(y)} \Delta\frac1{\sqrt{b(x)}} \text{ in }\mathcal{D}'(\Omega), \\
 & G_{\Omega,b}(x,y)=0 \text{ for all }x\in \partial\Omega.
\end{aligned}\right.
\end{equation}
In the case of a vanishing topography, the existence of $S_{\Omega,b}(\cdot,y)$ is not obvious because $ \Delta\frac1{\sqrt{b(x)}}$ is not integrable.

\begin{lem}\label{source.Lp}
 Let $(\Omega,b)$ verifying \eqref{b-cond} and let $y\in \Omega$. There exists a unique solution $S_{\Omega,b}(\cdot,y)\in X_{b}(\Omega)$ to \eqref{eq.GreenSb}. Moreover, for any $\delta>0$, there exists $C_{\delta}>0$ which depends only on $\Omega,b$ and $\delta$ such that
 \[
 \Big\| \frac1{\sqrt{b}} \nabla_{x}S_{\Omega,b}(\cdot,y) \Big\|_{L^2(\Omega)} \leq C_{\delta}, \quad \text{for all $y\in \Omega$ such that }{\rm dist}(y,\partial\Omega)\geq \delta.
 \]
 \end{lem}
\begin{proof}
For $\delta>0$ and $y\in \Omega$ fixed such that ${\rm dist}(y,\partial\Omega)\geq \delta$, we introduce the functional
\[
E(\psi)=\int_{\Omega}\Big( \frac1{2b(x)}|\nabla \psi(x)|^2- \sqrt{b(y)} \nabla\frac1{\sqrt{b(x)}} \cdot \nabla_{x}(G_{\Omega}(x,y) \psi(x) )\Big)\, dx.
\]
The first step is to show that $E$ is well defined on $X_{b}(\Omega)$, i.e., we start by studying the following integrals
\begin{align*}
 I_{1}&=\int_{\Omega} \psi(x) \nabla\frac1{\sqrt{b(x)}} \cdot \nabla_{x}G_{\Omega}(x,y) \, dx, \\
 I_{2}&=\int_{\Omega} G_{\Omega}(x,y) \nabla\frac1{\sqrt{b(x)}} \cdot \nabla \psi(x) \, dx.
\end{align*}

The first integral corresponds to the usual Biot-Savart law in a bounded simply connected open set. By the Cauchy-Riemann equations and the $C^1$ regularity of the Riemann mapping, we have
\begin{align*}
 I_{1}=&\frac1{2\pi}\int_{\Omega}\psi(x) \nabla\frac1{\sqrt{b(x)}} \cdot D\mathcal{T}^T(x) \Big( \frac{\mathcal{T}(x)-\mathcal{T}(y)}{|\mathcal{T}(x)-\mathcal{T}(y)|^2}-\frac{\mathcal{T}(x)-\mathcal{T}(y)^*}{|\mathcal{T}(x)-\mathcal{T}(y)^*|^2} \Big) \, dx\\
 |I_{1} | \leq & C_{\mathcal{T}}\int_{\Omega} \Big( \frac{F(x)}{|\mathcal{T}(x)-\mathcal{T}(y)|} + \frac{F(x)}{|\mathcal{T}(x)-\mathcal{T}(y)^*|} \Big) \, dx
\end{align*}
where $F = |\psi\nabla\frac1{\sqrt{b}} |$. By the continuity of $\mathcal{T}$, there exists $C_{\delta}$ such that
\[
|\mathcal{T}(x)-\mathcal{T}(y)^*| \geq |\mathcal{T}(y)^*|-1\geq C_{\delta}\quad \forall (x,y)\in \Omega^2\text{ such that }{\rm dist}(y,\partial\Omega)\geq \delta.
\]
For the first-right hand side term, we split the integral into the integral on $B(y,\delta/2)$ and its complement $\Omega\setminus B(y,\delta/2)$. In the disk, we use that ${\rm dist}(x,\partial\Omega)\geq \delta/2$, hence $|\nabla\frac1{\sqrt{b}}|$ is bounded:
\[
\int_{\Omega\cap B(y,\delta/2)}\frac{F(x)}{|\mathcal{T}(x)-\mathcal{T}(y)|} \, dx\leq C_{\delta, \mathcal{T}} \|\psi\|_{L^4} .
\]
In the exterior of the disk 
\begin{equation}\label{ineTxTy}
\delta/2\leq |x-y| \leq \|D\mathcal{T}^{-1}\|_{L^\infty} |\mathcal{T}(x)-\mathcal{T}(y)| 
\end{equation}
so 
\begin{align*}
|I_{1}|&\leq C_{\delta, \mathcal{T}} \Big( \|\psi\|_{L^4} + \int_{\Omega} |\psi\nabla\frac1{\sqrt{b}} | \Big)\\
&\leq C_{\delta, \mathcal{T}} \Big( \|\psi\|_{H^1} + \Big\|b^{-1/2} \frac{\psi}d\Big\|_{L^2(\Omega)} \Big\| b^{1/2} d \nabla\frac1{\sqrt{b}} \Big\|_{L^2(\Omega)} \Big) \\
&\leq C_{\delta, \mathcal{T}} \Big( \|\nabla \psi\|_{L^2} + \Big\| b^{-1/2} \nabla \psi \Big\|_{L^2(\Omega)} \Big\| \frac{d^{1+\frac\alpha2}}{d^{1+\frac\alpha2} } \Big\|_{L^2(\Omega)} \Big)\\
&\leq C_{\delta,\Omega,b} \Big\| b^{-1/2} \nabla \psi \Big\|_{L^2(\Omega)}
\end{align*}
where we have used the embedding of $H^1$ in $L^4$, the Poincar\'e inequality in $\Omega$, $b$ bounded, and the Hardy inequality \cite[Lem. 2.2]{la-pau}.

For $I_{2}$, we first notice by an easy computation that
\[
|X-Y^*|^2|Y|^2= |X-Y|^2+ (1-|X|^2)(1-|Y|^2)
\]
which gives a sense to $G_{D}$ even if $Y=0$. This relation will be used several times in the sequel. 
 We again split $\Omega$ into two regions $\Omega=B(y,\delta/2) \cup (\Omega\setminus B(y,\delta/2))$.
 
As $\mathcal{T}$ maps from $\Omega$ to $B(0,1)$, we have
\[
|\mathcal{T}(x)-\mathcal{T}(y)^*|^2 |\mathcal{T}(y)|^2 \leq 5.
\]
This implies that
\[
| G_{\Omega}(x,y) | = \frac1{4\pi} \ln \frac{|\mathcal{T}(x)-\mathcal{T}(y)^*|^2|\mathcal{T}(y)|^2}{|\mathcal{T}(x)-\mathcal{T}(y)|^2} \leq \frac1{4\pi} \ln 5 - \frac1{2\pi} \ln |\mathcal{T}(x)-\mathcal{T}(y)|
\]
which belongs to $L^p(B(y,\delta/2))$ for any $p\in [1,\infty)$. As $ \nabla\frac1{\sqrt{b(x)}}$ is bounded in $B(y,\delta/2)$, it is clear that 
\[
\int_{\Omega\cap B(y,\delta/2)} |G_{\Omega}(x,y)| \Big| \nabla\frac1{\sqrt{b(x)}} \cdot \nabla \psi(x) \Big| \, dx\leq C_{\delta, \mathcal{T}} \|\nabla \psi\|_{L^2} .
\]

Outside $B(y,\delta/2)$, we use \eqref{ineTxTy} to deduce
\begin{align*}
| G_{\Omega}(x,y) | &= \frac1{4\pi} \ln \Big(1+ \frac{(1-|\mathcal{T}(x)|^2)(1-|\mathcal{T}(y)|^2) }{|\mathcal{T}(x)-\mathcal{T}(y)|^2}\Big)\\
& \leq \frac1{4\pi} \frac{(1-|\mathcal{T}(x)|^2)(1-|\mathcal{T}(y)|^2) }{|\mathcal{T}(x)-\mathcal{T}(y)|^2} \leq C_{\delta,\Omega,b} d(x).
\end{align*}
where we have used that $\mathcal{T}$ is $C^1$ up to the boundary. This allows us to write
\begin{align*}
|I_{2}|&\leq C_{\delta, \mathcal{T}} \Big( \|\nabla \psi\|_{L^2} + \int_{\Omega}\Big|d(x) \nabla\frac1{\sqrt{b(x)}} \cdot \nabla \psi(x) \Big| \Big)\\
&\leq C_{\delta, \mathcal{T}} \Big( \|\nabla \psi\|_{L^2} + \Big\| b^{-1/2} \nabla \psi \Big\|_{L^2(\Omega)} \Big\| \frac{d^{1+\frac\alpha2}}{d^{1+\frac\alpha2} } \Big\|_{L^2(\Omega)} \Big)\\
&\leq C_{\delta,\Omega,b} \Big\| b^{-1/2} \nabla \psi \Big\|_{L^2(\Omega)}.
\end{align*}

Putting together the estimates of $I_{1}$ and $I_{2}$, we conclude that $E(\psi)$ is well defined for any $\psi\in X_{b}(\Omega)$. More precisely, there exists $C_{\delta}$ depending only on $\Omega,b,\delta$ such that for all $y\in \Omega$ verifying that ${\rm dist}(y,\partial\Omega)\geq \delta$ and all $\psi\in X_{b}(\Omega)$ we have
\[
\Big| E(\psi) - \frac12\Big\| b^{-1/2} \nabla \psi \Big\|_{L^2(\Omega)}^2 \Big| \leq C_{\delta} \Big\| b^{-1/2} \nabla \psi \Big\|_{L^2(\Omega)}.
\]
So
\[
 \frac14\Big\| b^{-1/2} \nabla \psi \Big\|_{L^2(\Omega)}^2 \leq E(\psi)+ C_{\delta}^2 \quad\text{and}\quad E(\psi) \leq \frac34\Big\| b^{-1/2} \nabla \psi \Big\|_{L^2(\Omega)}^2 + C_{\delta}^2 .
\]

This inequality allows to adapt the standard Lax-Milgram type argument, which also gave the existence of a solution of \eqref{ellip-eq}: Let $\psi_{k}$ be a minimizing sequence, $\psi_{k}$ is uniformly bounded in $X_{b}(\Omega)$. Up to a subsequence, we assume that $\psi_{k}\rightharpoonup \psi$ weakly in $X_{b}(\Omega)$. By the lower semi-continuity of the norm, it follows that $E(\psi)\leq \liminf_{k\to\infty} E(\psi_{k})$, so $\psi \in X_{b}(\Omega)$ is indeed a minimizer. In addition, by minimization, the first variation of $E(\psi)$ reads
\[
\int_{\Omega}\Big( \frac1{b(x)}\nabla \psi(x) \cdot \nabla \varphi - \sqrt{b(y)} \nabla\frac1{\sqrt{b(x)}} \cdot \nabla_{x}(G_{\Omega}(x,y) \varphi(x) )\Big)\, dx,
\quad \forall \varphi\in C^\infty_{c}(\Omega),
\]
which shows that $S_{\Omega,b}(\cdot,y):=\psi$ is a solution of \eqref{eq.GreenSb}. We recall that the Dirichlet boundary condition is encoded in the function space $X_{b}(\Omega)$. The uniqueness follows from the uniqueness of \eqref{ellip-eq} when $f\equiv 0$.

Fixing any $\psi_{0}\in C^\infty_{c}(\Omega)$, we have proved that
\[
 \frac14\Big\| b^{-1/2} \nabla_{x} S_{\Omega,b}(\cdot,y) \Big\|_{L^2(\Omega)}^2 \leq E(\psi_{0})+ C_{\delta}^2,\quad \forall y\in \Omega \text{ such that } {\rm dist}(y,\partial\Omega)\geq \delta.
\]
This ends the proof of this lemma. 
\end{proof}

Thanks to the existence of $(x,y)\mapsto S_{b}(x,y)$, we can verify that for any $f\in C^\infty_{c}(\Omega)$, the function $\Phi(x)$ defined below
\[
\Phi (x):= G_{\Omega,b}[f](x) := \int_{\Omega} \Big(G_{\Omega}(x,y)\sqrt{b(x)}\sqrt{b(y)} + S_{\Omega,b}(x,y)\Big) f(y)\, dy
\]
is a solution of \eqref{ellip-eq}. Indeed, it is clear that $\Phi$ satisfies the Dirichlet boundary condition. For $\delta>0$ such that $\dist(\supp f,\partial\Omega)>\delta$, we use Lemma~\ref{source.Lp} to state that 
\[
x\mapsto \int_{\Omega} \frac1{\sqrt {b(x)}} \nabla_{x}S_{\Omega,b}(x,y) f(y)\, dy \in L^2(\Omega).
\]
Regarding the estimates made for $I_{1}$ and $I_{2}$ in the previous proof, we have
\begin{align*}
& \Big|\int_{\Omega} \Big( \nabla_{x} G_{\Omega}(x,y)+ G_{\Omega}(x,y)\frac{\nabla b(x) }{2 b(x)} \Big)\sqrt{b(y)} f(y)\, dy\Big|\\
 &\leq C_{\mathcal{T}} \int_{\Omega} \frac{|\sqrt{b}f|(y)}{|x-y|} \, dy+C_{\delta,\mathcal{T}} \|\sqrt{b}f\|_{L^1}+C_{\delta,b,\mathcal{T}} \|f\|_{L^2}+C_{\delta,\mathcal{T}}\frac{d(x)|\nabla b(x)| }{|b(x)|} \|\sqrt{b}f\|_{L^1}\\
 &\leq C_{\delta,b,\mathcal{T}} \|f\|_{L^4}
\end{align*}
as $\frac{d(x)|\nabla b(x)| }{|b(x)|}$ is uniformly bounded. This allows to conclude that $\frac1{\sqrt{ b(x)}} \nabla \Phi (x)$ belongs to $L^2(\Omega)$. And finally, we verify that for any $\psi\in C^\infty_{c}(\Omega)$
\begin{align*}
 \Big\langle \divv \frac{ \nabla \Phi}{b} , \psi \Big\rangle
 =& -\iint_{\Omega^2} \frac{\sqrt{b(y)}}{\sqrt{b(x)}} \Big(\nabla_{x} G_{\Omega}(x,y)+ G_{\Omega}(x,y)\frac{\nabla b(x) }{2 b(x)} \Big)\cdot \nabla \psi(x) f(y)\, dydx\\
 &-\iint_{\Omega^2} \frac1{ b(x)} \nabla_{x}S_{\Omega,b}(x,y) \cdot \nabla \psi(x) f(y)\, dydx\\
 =& -\iint_{\Omega^2} \frac{\sqrt{b(y)}}{\sqrt{b(x)}} \Big(\nabla_{x} G_{\Omega}(x,y)+ G_{\Omega}(x,y)\frac{\nabla b(x) }{2 b(x)} \Big)\cdot \nabla \psi(x) f(y)\, dxdy\\
& - \iint_{\Omega^2} \sqrt{b(y)} \nabla\frac1{\sqrt{b(x)}} \cdot \nabla_{x}(G_{\Omega}(x,y) \psi(x)) f(y)\, dxdy\\
=& -\int_{\Omega} \sqrt{b(y)} f(y)\int_{\Omega} \nabla_{x} G_{\Omega}(x,y)\cdot \nabla\Big( \frac1{\sqrt{b(x)}}\psi(x)\Big) \, dxdy\\
=& \int_{\Omega}f(y)\psi(y)\, dy
\end{align*}
which means that $\Phi$ is a solution of \eqref{ellip-eq}. As it is clear that for any $\psi,\varphi\in C^\infty_{c}(\Omega)$
\begin{multline*}
\int_{\Omega} \psi(x)G_{\Omega,b}[\varphi](x) \, dx= \Big\langle \divv \frac1{b} \nabla G_{\Omega,b}[\psi] , G_{\Omega,b}[\varphi]\Big\rangle\\ 
= -\int_{\Omega }\frac1b \nabla G_{\Omega,b}[\psi] \cdot \nabla G_{\Omega,b}[\varphi] = \int_{\Omega} \varphi(x)G_{\Omega,b}[\psi](x) \, dx ,
\end{multline*}
so we deduce by the symmetry of $(x,y)\mapsto G_{\Omega}(x,y)\sqrt{b(x)}\sqrt{b(y)}$ that
\[
\int_{\Omega}\int_{\Omega} \psi(x)\varphi(y) S_{\Omega,b}(x,y) \, dxdy=\int_{\Omega}\int_{\Omega} \psi(y)\varphi(x) S_{\Omega,b}(x,y) \, dxdy
\]
which reads as $S_{\Omega,b}(x,y)=S_{\Omega,b}(y,x)$. This ends the proof that 
\begin{equation}\label{Green-decomp}
G_{\Omega,b}(x,y):=G_{\Omega}(x,y)\sqrt{b(x)}\sqrt{b(y)}+S_{\Omega,b}(x,y)
 \end{equation}
is the green kernel \eqref{eq.Green} associated to the elliptic problem \eqref{ellip-eq}.

\medskip

The decomposition \eqref{Green-decomp} will be useful after proving that the remainder $x\mapsto S_{\Omega,b}(x,y)$ is more regular than $G_{\Omega,b}(x,y)$.

\begin{lem}\label{S.reg}
 Let $(\Omega,b)$ verifying \eqref{b-cond}. Let $K,K'\Subset \Omega$, then there exist $C_{K,K'}>0$ such that for any $f\in L^\infty(\Omega)$ compactly supported in $K'$, the function
 \[
 \tilde\Phi(x):=\int_{\Omega}S_{\Omega,b}(x,y)f(y)\, dy
 \]
 is of class $C^2(\Omega)$ and 
 \[
 \| \tilde\Phi \|_{W^{2,\infty}(K)} \leq C_{K,K'} \|f\|_{L^\infty(\Omega)}.
 \]
 \end{lem}

\begin{proof}
Let $\tilde K$ such that $K\Subset \tilde K\Subset \Omega$.
As $b(x)\geq C_{K}>0$ on $\tilde K$, we use \eqref{eq.GreenSb} to get for all $x\in \tilde K$
\[
\divv\Big(\frac1{b(x)}\nabla \tilde\Phi(x)\Big)= \Delta\frac1{\sqrt{b(x)}} \int_{\Omega} G_{\Omega}(x,y) \sqrt{b(y)}f(y)\, dy=:F(x).
\]
On $\tilde K$, the elliptic problem $\divv(b^{-1} \nabla\cdot )$ is non singular, and standard elliptic estimate gives for any $p>2$
\[
 \| \tilde\Phi \|_{W^{2,\infty}(K)} \leq C_{p} \| \tilde\Phi \|_{W^{3,p}(K)} \leq C_{K,p} \Big( \|F\|_{W^{1,p}(\tilde K)} +\| \tilde\Phi \|_{L^{2}(\tilde K)} \Big),
\]
and that $ \tilde\Phi\in C^2(\Omega)$, provided that we prove that $F\in W^{1,p}_{\rm loc}(\Omega)$.
The $L^2$ estimate of $ \tilde\Phi$ comes directly from Lemma~\ref{source.Lp} and the Poincar\'e inequality:
\begin{align*}
 \| \tilde\Phi \|_{L^{2}(\Omega)} &\leq C_{\Omega}\int_{\Omega} \| \nabla_{x} S_{\Omega,b}(\cdot,y) \|_{L^{2}(\Omega)} |f(y)|\, dy\\
 & \leq C_{\Omega,b} \int_{\Omega} \| b^{-1/2} \nabla_{x} S_{\Omega,b}(\cdot,y) \|_{L^{2}(\Omega)} |f(y)|\, dy\leq C_{K'}\|f\|_{L^\infty(\Omega)}.
\end{align*}
Concerning the $W^{1,p}$ estimate of $F$, we use that $b\in C^3$ and $b(x)\geq C_{K}>0$ on $\tilde K$ together with the remark that $\tilde F(x):= \int_{\Omega} G_{\Omega}(x,y) \sqrt{b(y)}f(y)\, dy $ is the solution of the classical Laplace's equation $\Delta \tilde F=\sqrt b f$ with Dirichlet boundary condition, hence veryfying
\[
\Big\| \int_{\Omega} G_{\Omega}(\cdot,y) \sqrt{b(y)}f(y)\, dy\Big\|_{W^{2,p}(\Omega)} \leq C_{p}\|\sqrt b f\|_{L^p(\Omega)}\leq C\|f\|_{L^\infty(\Omega)}.
\]
\end{proof}

The regularity of this remainder allows us to adapt standard result for the Biot-Savart kernel in $\R^2$.

\begin{prop}\label{est.local}
 Let $(\Omega,b)$ verifying \eqref{b-cond} and $K\Subset \Omega$. There exists $C_{K}>0$, and for any $M>0$ there exists $C_{M,K}>0$ such that the following assertions hold.
\medskip

(i) {\bf Local log-Lipschitz regularity}. For a given $f \in L^{\infty}(\Omega)$, the vector field $u=\frac1b\nabla^{\perp}\psi$ where $\psi$ is the solution of \eqref{ellip-eq} is log-Lipschitz on $K$, more precisely
\begin{equation*}
|u(x)-u(y)| \leq C_{K} \|f\|_{L^\infty(\Omega)} |x-y| \Big(1+\big|\ln |x-y| \big|\Big), \quad \forall x,y\in K.
\end{equation*}

\medskip

(ii) {\bf Local $C^1$ regularity}. In addition, if $f \in C^{1}(\Omega)$ such that $\|f\|_{L^\infty(\Omega)}\leq M$, then the vector field $u$ belongs to $C^{1,\beta}(K)$ and \eqref{C1-ln} holds.
\end{prop}

\begin{proof}
By local elliptic regularity \cite{ag-do-ni1, ag-do-ni2}, it is clear that $f\in C^{1}(\Omega)$ implies that $u\in C^{1,\beta}(K)$.

We fixed $K'$ such that $K\Subset K'\Subset \Omega$, $\chi\in C^\infty_{c}(\Omega)$ such that $\chi=1$ on $K'$, and we consider $u_{\rm int}$ and $u_{\rm ext}$ associated to $f_{\rm int}=f \chi$ resp. $f_{\rm ext}=f (1-\chi)$. By linearity, we have $u=u_{\rm int}+u_{\rm ext}$.

The vector field $u_{\rm ext}$ is $b$-harmonic on $K'$, so by local elliptic regularity, we have
\[
\| \nabla u_{\rm ext} \|_{L^\infty(K)} \leq C_{K} \| u_{\rm ext} \|_{L^2(K')} \leq C_{K} \|f\|_{L^\infty(\Omega)}
\]
where we have used \eqref{energy}. Points {\rm (i)} and {\rm (ii)} are then obvious for this part.

For the other part, we use the decomposition of the Green kernel \eqref{Green-decomp}:
\begin{multline*}
 u_{\rm int}(x)= \frac1{b(x)} \int_{\Omega} \nabla^\perp_{x} \Big( G_{\Omega}(x,y)\sqrt{b(x)} \Big)\sqrt{b(y)} f(y)\chi(y)\, dy \\
 +\frac1{b(x)}\int_{\Omega} \nabla_{x}^\perp S_{\Omega,b}(x,y) f(y)\chi(y)\, dy.
\end{multline*}
By Lemma~\ref{S.reg}, the last integral on the right-hand side term belongs to $C^1$ and its $W^{1,\infty}(K)$ norm is bounded by $C_{K}\|f\|_{L^\infty}$. Points {\rm (i)} and {\rm (ii)} are then also clear for this term. In the first integral, we recall that there is a part which depends on $|\mathcal{T}(x)-\mathcal{T}(y)^*|$ which gives $C^1$ contribution because $|\mathcal{T}(x)-\mathcal{T}(y)^*|\geq C_{K}$ for every $(x,y)\in K\times \Omega$. Hence, the only singular term is
\begin{multline*}
 \frac1{b(x)} \int_{\Omega} \nabla^\perp_{x} \Big( \ln |\mathcal{T}(x)-\mathcal{T}(y)| \sqrt{b(x)} \Big)\sqrt{b(y)} f(y)\chi(y)\, dy\\
=\frac1{b(x)} \int_{\R^2} \nabla^\perp_{x} \Big(\ln|\mathcal{T}(x)-\xi | \sqrt{b(x)} \Big) g(\xi)\, d\xi
\end{multline*}
with
\[
g(\xi):= (\sqrt{b}f\chi)(\mathcal{T}^{-1}(\xi)) |\det D\mathcal{T}^{-1}(\xi)|
\]
a bounded function compactly supported in $B(0,1)$. The most singular term is a composition of $\nabla^\perp \ln|x-\cdot|\ast g$ by $\mathcal{T}$ regular, times some powers of $\sqrt{b(x)}$ which are regular on $K$. The usual regularity for the Green kernel in $\R^2$ ends the proof: we refer for instance to \cite[App. 2.3]{MarPul} for the log-Lipschitz regularity and to \cite[Lem. 7.2]{ADL} for the $C^1$ estimate in terms of $\ln \| g\|_{C^1}$.
\end{proof}

\subsection{Elliptic regularity up to the boundary}

To continue the proof of Theorem~\ref{ellip-est}, we have now to focus on high regularity up to the boundary.

First, let us notice that the question of the elliptic regularity up to the boundary for equation \eqref{ellip-eq} 
has been already addressed in \cite{br-me-lake}, and so the estimate \eqref{ell-reg} is mainly a consequence of the results proved in that paper. We recall that the authors in \cite[Theo. 2.3]{br-me-lake} showed
that $u$ is continuous (in $C^\beta$ for any $\beta\in [0,1)$), belongs to any $W^{1,p}$ and we have
\begin{gather*}
 \|u \|_{L^\infty} \leq C_{0}\big(\|f\|_{L^\infty} + \| bu \|_{L^2} \big)\\
\frac1p \|\nabla u \|_{L^p} \leq C_{0}\big(\|f\|_{L^p} + \| bu \|_{L^2} \big)\quad \forall p\in [p_{0},\infty)\quad p_0>2,
\end{gather*}
where $C_{0}$ depends only on $\Omega,b,p_{0}$. Using that the domain is bounded and \eqref{energy}, it is clear that \eqref{ell-reg} is an obvious consequence of \cite[Theo. 2.3]{br-me-lake}.

By Proposition~\ref{est.local}, to complete the proof of Theorem~\ref{ellip-est}, we only need to show that $u$ is log-Lipschitz in a neighborhood of the boundary with an estimate as in \eqref{log-lip} and that $u$ is $C^{1,\beta}$ up to the boundary if $f\in C^1$. 
For such a regularity close to the boundary, the decomposition \eqref{Green-decomp} is not convenient because of the singular behavior of $\Delta b^{-1/2}$ in \eqref{eq.GreenSb}. A possible approach to follow is the one introduced by {\sc E. B. Fabes} and his coauthors in \cite{Fa-Je-Se, Fa-ke-Se1} where they were interested in a similar problem. Actually, in a part of their results, the authors studied the behavior of the Green function associated to \eqref{ellip-eq} when the coefficient $1/b(x)$ is replaced by $h(x)$ and the function $h(x)$ is in the class $\mathcal{A}_2$ of Muckenhoupt weights. Following their works, one could prove the existence of a Green function of \eqref{ellip-eq} and derive estimates for its size. However, this will lead us to assume a restrictive condition on our function $b$. For instance, if $b$ is assumed to be as $\dist(x, \del \Omega)^\alpha$, then this will lead to the following assumption, $0<\alpha<1$. For this reason we choose to follow the analysis of {\sc Bresch} and {\sc M\'etivier} performed in \cite{br-me-lake}.

Before starting the proof, we introduce some notations in $\mathbb{\R}^2_+=\{ (x_{\tau},x_{n})\in \R\times \R_{+}^*\}$. For $\beta\in (0,1)$, we denote by $C^{\beta}(\overline{\mathbb{\R}^2_+})$ the space of bounded functions on $\mathbb{\R}^2_+$ which are uniformly H\"older continuous with exponent $\beta$. Moreover, we denote by $C^\beta_1(\overline{\mathbb{\R}^2_+})$ the space of functions $u\in C^\beta(\overline{\mathbb{\R}^2_+})$ such that $x_n u\in C^{\beta+1}(\overline{\mathbb{\R}^2_+})$. Similarly, we denote by $W^{k,p}_1(\Omega)$ the space of functions $u$ in $W^{k,p}(\Omega)$ such that $x_n u\in W^{k+1,p}(\Omega)$. In this section, we consider $b=\varphi^\alpha$ with $\varphi(x)=c(x)\dist(x,\partial\Omega)$ where $c(x)\geq c_{0}>0$.

The analysis performed in \cite{br-me-lake} is based on rewriting \eqref{ellip-eq} in a different form with the use of a new variable $\Phi=\varphi^{-(\alpha+1)}\psi$. We can easily check that $\Phi$ satisfies the following equation
\begin{equation}\label{eq-phi}
 \varphi \Delta \Phi +(\alpha+2) \nabla \varphi \cdot \nabla \Phi +(\alpha+1) \Delta \varphi\, \Phi =f.
\end{equation}
As we are interested in establishing estimates near the boundary, then we proceed by using local coordinates. The boundary $\del \Omega$ is assumed to be a closed smooth manifold of class $C^3$. Consider a coordinate patch $x_\tau \rightarrow \gamma(x_\tau)$ from an open interval $w \subset \mathbb{R}$ to $\del \Omega$, with $\gamma\in C^3$ on $\overline{w}$. Taking $\nu(x_\tau)$ to be the inward unit normal to $\del \Omega$ at $\gamma(x_\tau)$, 
we parametrize a neighbourhood $V$ of $\gamma(w)$ by $(x_\tau, x_n)\in w \times (-\delta,\delta)$ considering the mapping
\begin{align*}
\Gamma: \; w \times (-\delta,\delta) \; &\longrightarrow V\\
(x_\tau, x_n)\; &\longmapsto \gamma(x_\tau)+x_n \nu(x_\tau).
\end{align*}
Due to the $C^3$ regularity of $\del\Omega,$ $\Gamma$ is a $C^3$ diffeomorphism, which is chosen to preserves the normal direction. We write $x=(x_1,x_2)=\Gamma(x^*)$ and $x^*=(x_\tau, x_n)=\Gamma^{-1}(x)=J(x).$ Set
\begin{align*}
&\tilde{\Phi}(x^*)=\Phi(\Gamma(x^*))\quad \mbox { for}\quad x^*\in w \times (-\delta,\delta), \\
&\tilde{\Phi}(J(x))=\Phi(x)\;\qquad \mbox {for}\quad x\in V.
\end{align*}
For $x^*\in w \times [0,\delta)$, it is not difficult to show that $\tilde{\varphi}(x^*)=\varphi(\Gamma(x^*))=x_n\tilde c(x^*)$ with $\tilde c\in C^3 (w \times [0,\delta))$ such that $\tilde c\geq \tilde c_{0}>0$. Moreover, we use the Chain Rule formula to compute
\begin{align*}
\dfrac{\del \Phi}{\del x_j}= \dfrac{\del \tilde{\Phi}}{\del x_\tau}\dfrac{\del J_1}{\del x_j}+ \dfrac{\del \tilde{\Phi}}{\del x_n} \dfrac{\del J_2}{\del x_j},
\end{align*}
and consequently, we deduce
\begin{align*}
\dfrac{\del^2 \Phi}{\del x_j^2}=\dfrac{\del \tilde{\Phi}}{\del x_\tau}\dfrac{\del^2 J_1}{\del x_j^2}+ \dfrac{\del \tilde{\Phi}}{\del x_n} \dfrac{\del^2 J_2}{\del x_j^2}+\dfrac{\del^2 \tilde\Phi}{\del x_\tau^2} \bigg| \dfrac{\del J_1}{\del x_j}\bigg|^2 + \dfrac{\del^2 \tilde\Phi}{\del x_n^2} \bigg| \dfrac{\del J_2}{\del x_j}\bigg|^2+2\dfrac{\del^2 \tilde\Phi}{\del x_\tau \del x_n} \dfrac{\del J_1}{\del x_j}\dfrac{\del J_2}{\del x_j}.
\end{align*}
Thus, Equation~\eqref{eq-phi} in these new coordinates becomes (omitting the tilde for simplicity)
\begin{equation}\label{Eq-Phi}
 \mathcal{L}(\Phi)=x_n c(x) \sum_{j,k=1}^2 b_{jk}(x)\del_j \del_k \Phi+\sum_{j=1}^2 b_{j}(x)\del_j \Phi+b_0 (x)\Phi=f,
\end{equation}
where we denote by
\begin{align*}
 & b_{jk}=\sum_{i=1}^2 \del_{x_i} J_j \del_{x_i}J_k\\
 & b_j= x_n c c_j +(\alpha+2)d_j \qquad c_j:= \sum_{i=1}^2\del^2_{x_i}J_j \qquad d_j:= \sum_{i=1}^2 \del_{x_i}\varphi b_{ji}\\
 & b_0=(\alpha+1)\bigg[\sum_{j,k=1}^2 b_{jk}(x) \del_j \del_k \varphi +\sum_{j=1}^2 c_j\del_j \varphi\bigg].
\end{align*}
One can check that the following properties hold:
\begin{itemize}
\item the coefficients $b_{jk}$ are real and we have $b_{jk}=b_{kj}$;
\item there exists a constant $C$ such that for all $x\in w \times (0,\delta)$ and $\xi \in \mathbb{R}^2,$ we have
\[
c(x) \sum_{j,k=1}^2 b_{jk}(x)\xi_j \xi_k \geq C |\xi|^2;
\]
\item for $j=1,2$, we have $b_{22}b_j -b_2 b_{2j}=\mathcal{O}(x_n)$;
\item we have $b_2>0$, for $x_n=0$. 
\end{itemize}
Equation \eqref{Eq-Phi} is a particular case of degenerate elliptic equations studied in \cite{br-me-lake,gou-shima}. The main results shown in these papers are the H\"older and $W^{2,p}$ estimates of solution $\Phi$ for a given function $f$ in $L^p$.
From \cite{br-me-lake}, we recall that if $\psi \in H^1_0(\Omega)$ and $f\in L^{p}(\Omega)$ with $p>2,$ then $\Phi=\varphi^{-(\alpha+1)}\psi$ satisfies the following estimate on H\"older spaces for $\beta=1-\frac2p$
\begin{equation}\label{Holder-est}
\| \Phi \|_{C^{\beta}(\overline{w_1}\times [0, \delta])}+\|x_n \Phi \|_{C^{\beta+1}(\overline{w_1}\times [0, \delta])}\leq C_\beta \big( \| f \|_{L^p(\Omega)}+ \| \psi \|_{H^1(\Omega)}\big),
\end{equation}
for all relatively compact subsets $w_1\subset w,$ as well the $L^p-$estimates
\begin{equation}\label{Lp-est}
\| \del_j \Phi \|_{L^{p}(\R^2_+)}+\| x_n \del_j \del_k \Phi \|_{L^{p}(\R^2_+)} \leq C p \big( \| f \|_{L^p(\Omega)}+ \| \psi \|_{H^1(\Omega)}\big).
\end{equation}
As $u= \varphi^{-\alpha}\nabla^\perp \psi= \varphi \nabla^\perp \Phi +(\alpha+1) \Phi \nabla^\perp \varphi$, these estimates give the $C^\beta$ and $W^{1,p}$ regularity for $u$, where we recall \eqref{energy} and that $\varphi$ behaves as $x_{n}$ by the straightening of the boundary. The estimate \eqref{Holder-est} is deduced from the results established in \cite{BCM} for a more general equation, verifying the properties listed above, see Sections 4-5 in \cite{br-me-lake} for more details. Notice that if we consider $f\in C^{1}(\overline{\Omega})$, then the right-hand side term in \eqref{Eq-Phi} becomes in $C^{1}(\R^2_+)$, particularly, in $C^{\beta}(\R^2_+)$ for any $\beta< 1$. Theorem~1 in \cite{gou-shima} implies that
\begin{equation*}
 \Phi \in C^{\beta+1}(\R^2_+)\qquad x_n \Phi \in C^{\beta+2}(\R^2_+),
\end{equation*}
and this proves that $ u\in C^{1,\beta}(\R^2_+),$ for any $\beta<1$. This fact together with Proposition~\ref{est.local} finishes the proof of (ii) in Theorem~\ref{ellip-est}.
\begin{rmk}
 The regularity $C^1$ for $u$ developed in the previous argument, totally based on \cite{gou-shima}, only gives $\|\nabla u\|_{L^\infty (\Omega)} \leq C \| \omega\|_{W^{1,\infty}(\Omega)}$. For the fixed point procedure in Section~\ref{sec-pointfixed}, it is crucial to have an estimate of $\|\nabla u\|_{L^\infty}$ by $F(\|\nabla \omega\|_{L^\infty})$ where $F$ is sublinear. Replacing $C \| \nabla f\|_{L^\infty(\Omega)}$ by $\ln(2+\|\nabla f\|_{L^\infty})$ was exactly the purpose of Proposition~\ref{est.local}.
\end{rmk}
In \cite{br-me-lake}, the proof of estimate \eqref{Lp-est} was more technical and it required a careful analysis of kernel estimates of the two operators $\mathbb{E}$ and $\mathbb{K}$ given below. As these kernel estimates constitute the main ingredients of our proof of the log-Lipschitz estimate, we detail this part in the sequel. The reader is referred to Section 6 in \cite{br-me-lake}, where all these estimates in full details were shown. 

Before stating and proving the desired log-Lipschitz estimate, we recall here a known result about operators with singular kernels.
\begin{prop}\label{prop-MarPul}
 Suppose that kernel $K(x,y)$ satisfies on $\R^2_{+}\times\R^2_{+}$:
 \[
 |K(x,y)|\leq \frac C{|x-y|}, \quad |\partial_{x} K(x,y)|\leq \frac C{|x-y|^2}.
 \]
 Then, the operator
 \[
 Tf(x) = \int_{\R^2_{+}} K(x,y) f(y)\, dy
 \] 
 acts from $L^\infty_{c}(\R^2_{+})$ to ${\rm LogLip}(\R^2_+)$, where
\[
 \| F \|_{{\rm LogLip}(\R^2_+)} :=\|F\|_{L^\infty(\R^2_{+})}+ \sup_{x\neq y \in \R^2_+} \frac{|F(x)-F(y)|}{|x-y| (1+|\ln |x-y||) }.
\]
\end{prop}

\begin{proof}
We follow \cite[App. 2.3]{MarPul}. Indeed, the first estimate for $K$ implies that $T$ maps $L^\infty_{c}(\R^2_{+})$ to $L^\infty(\R^2_{+})$:
 \[ |Tf(x)| \leq C\int_{B(x,1)} \frac{|f(y)| }{|x-y|} \, dy + C\int_{B(x,1)^c} \frac{|f(y)| }{|x-y|} \, dy\leq C(\|f\|_{L^\infty}+\|f\|_{L^1}).
 \]
 
It is then enough to show the log-Lipschitz estimate for $x,x'\in \R^2_{+}$ such that $|x-x'|\leq1/2$. We set $\tilde x=(x+x')/2$ and $\delta=|x-x'|$, the second estimate for $K$ implies that
 \[
 |K(x,y)-K(x',y)|\leq C\frac{|x-x'|}{|\tilde x-y|^2} \quad \text{for all $y$ such that}\quad |\tilde x-y|\geq 2|x-x'|.
 \]
 We write 
\begin{align*}
 |Tf(x)-Tf(x')| \leq& C \int_{\R^2_{+}\setminus B(\tilde x, 2\delta)} \frac{|x-x'|}{|\tilde x-y|^2} |f(y) |\, dy + 
 C \int_{B(x, 3\delta)} \frac{|f(y)| }{| x-y|} \, dy \\
 &+ C \int_{B(x', 3\delta)} \frac{|f(y)| }{|x'-y|} \, dy \\
 \leq& C\delta \|f\|_{L^1(B(\tilde x,1)^c)} + C\delta \|f\|_{L^\infty(B(\tilde x,1))} |\ln(2\delta)| + C\delta \|f\|_{L^\infty}
\end{align*}
 which ends the proof of the proposition.
\end{proof}
We recall that the goal of this section is to add the following log-Lipschitz estimate:
\begin{equation}\label{LogLip-est}
\| \Phi \|_{{\rm LogLip}(\R^2_+)}+\| x_n \del_j \Phi \|_{{\rm LogLip}(\R^2_+)} \leq C \big( \| f \|_{L^\infty(\Omega)}+ \| \psi \|_{H^1(\Omega)}\big),
\end{equation}
from which, along with Proposition~\ref{est.local}, we can deduce \eqref{log-lip} in Theorem~\ref{ellip-est}.

In what follows, we shall prove estimate \eqref{LogLip-est} following the same lines performed in \cite{br-me-lake} to prove estimate \eqref{Lp-est}. Given $w_1,w_{2}$ such that $w_1\Subset w_2\Subset w$ and $\delta'\in (0,\delta)$, by estimate \eqref{Holder-est}, $\Phi$ is of class $C^\beta_1$ on $w_2 \times [0,\delta]$ for any $\beta\leq 1-2/p.$ Consider $\chi\in C_c^\infty(w_2\times[0,\delta))$ such that $\chi=1$ on $w_1\times[0,\delta']$. Let $\phi=\chi \Phi\in C^\beta_1(\R^2_+).$ Then we have
\begin{equation*}
\mathcal{L}\phi=g:=\chi f+[\mathcal{L}, \chi]\Phi.
\end{equation*}
We will use several times that $[\mathcal{L}, \chi]$ is of the form $x_n \sum_{j} A_{j}(x)\del_j +A_0 (x)$, where $A_0(\cdot)$ and $A_j(\cdot)$ are bounded continuous functions.
Since $\Phi \in C^\beta_1$, then it is obvious that $g \in L^{p}(\R^2_+)$ for any $p\leq \infty$. Next, by \eqref{Holder-est}, choosing $p_{0}> 2$ and $\mu_{0}=1-2/p_{0}$, the following estimates hold for any $p>p_{0}$ and $f\in L^p$:
\begin{equation*}
\begin{split}
 \| \phi \|_{C^{\mu_{0}}_{1}(\overline{\mathbb{R}^2_+})} \leq C \big( \| f \|_{L^{p_{0}}(\Omega)}+ \| \psi \|_{H^1(\Omega)}\big),\\
 \| g \|_{L^{p}(\mathbb{\R}^2_+)} \leq C \big( \| f \|_{L^{p}(\Omega)}+ \| \Phi \|_{C^{\mu_{0}}_{1}(\overline{\mathbb{R}^2_+})}\big) .
\end{split}
\end{equation*}
Estimate \eqref{Lp-est} is the consequence of \cite[Theo. 6.1]{br-me-lake} which states that
\begin{equation}\label{CZ}
 \| \del_j \phi \|_{L^{p}(\R^2_+)}+\| x_n \del_j \del_k \phi \|_{L^{p}(\R^2_+)} \leq C \big(p \| g\|_{L^p(\R^2_+)}+ \| \phi \|_{C^{\mu_{0}}_{1}(\overline{\R^2_+})}\big)
\end{equation}
for some $C$ independent of $p$. In the same way, we say that \eqref{LogLip-est} is an obvious corollary of the following theorem.
\begin{theo}\label{theo-BM-log}
 Suppose that $\phi$ has compact support in $w\times[0,\delta)$ and $\mathcal{L}\phi=g\in L^{\infty}(\R^2_+)$. Then $\phi$ is log-Lipschitz and there is $C$ such that for all such $\phi$
\begin{equation}\label{loc-log-lip}
\| \phi \|_{{\rm LogLip}(\R^2_+)}+\| x_n \del_j \phi \|_{{\rm LogLip}(\R^2_+)} \leq C \big( \| g \|_{L^\infty(\Omega)}+ \| \phi \|_{C^{\mu_{0}}_{1}(\overline{\R^2_+})}\big).
\end{equation} 
\end{theo}
The main idea in the proof of \eqref{CZ}, and also of \eqref{loc-log-lip}, is to compare the operator $\mathcal{L}(\cdot)$ to a model operator, say, $\mathcal{L}_{y_0}(\cdot)$ for which we know its explicit solution. Indeed, for any $y_{0}\in \overline{w}\times (0,\delta)$, we denote by $\mathcal{L}_{y_{0}}$ the operator 
\[
\mathcal{L}_{y_{0}} (x,\partial_{x}) := x_n \frac{c(y_{0})}{d_{2}(y_{0})} \sum_{j,k=1}^2 b_{jk}(y_{0})\del_j \del_k + (\alpha+2) \partial_{2} + \frac{b_{1}(y_{0})}{d_{2}(y_{0})} \del_1 .
\]
By a linear transformation, we can replace this operator by (see Section 6.2 in \cite{br-me-lake} for details) 
\[
\tilde{\mathcal{L}}=\tilde x_n \Delta_{\tilde x} + (\alpha+2) \partial_{\tilde x_2} ,
\] 
for which we introduce the fundamental solution $\tilde E$ and a regularization $\tilde E^\varepsilon$ for $\varepsilon\in (0,1)$. From these kernels, we derive the fundamental solution and the $\varepsilon$ approximation for $\mathcal{L}_{y_{0}}$: 
\begin{gather*}
 E_{y_{0}}(x,y):=|\det T'(y_{0}) | E(T(y_{0})x,T(y_{0})y),\\ 
 E^\varepsilon_{y_{0}}(x,y):=|\det T'(y_{0}) | E^\varepsilon(T(y_{0})x,T(y_{0})y).
 \end{gather*}
 Finally, one defines the parametrices
 \[
 E(x,y):=E_{y}(x,y) \quad \text{and}\quad E^\varepsilon(x,y):=E^\varepsilon_{y}(x,y).
 \] 
 On the next step, the authors decomposed $\mathcal{L} E^\varepsilon(x,y)$ as $G^\varepsilon(x,y)+K^\varepsilon(x,y)$ where the more regular part $K^\varepsilon$ is estimated in Lemma 6.6 whereas the explicit part $G^\varepsilon$ is studied in Lemma 6.7. For convenience, we state in the following lemma some results obtained in \cite{br-me-lake}, and we refer to Section 6 therein for the proof.

\begin{lem}\label{lem-E}
For $k \in \mathbb{N}$, there is a constant $C_k$ such that for
all $x,y \in w \times \mathbb{R}^+ $, $x \neq y$, there holds
\begin{equation*}
 |\nabla^k E(x,y)|\leq \dfrac{C_k}{|x-y|^{k+1}}, \qquad |x_n \nabla^{k+1} E(x,y)| \leq \dfrac{C_k}{|x-y|^{k+1}}.
\end{equation*}
Moreover, if we denote by $\mathbb{E}$ the operator with kernel $E$, then $\mathbb{E}$ maps from $L^p(w\times (0,\delta))$ to $W^{1,p}_{1,\rm loc}(\overline{\R^2_+})$, and, moreover, the following identity holds
\begin{equation*}
\mathcal{L}\mathbb{E}g -g =\mathbb{K}g\in C^\beta_{1,{\rm loc}}(\overline{\R^2_+}),
\end{equation*}
where the kernel of $\mathbb{K}$ satisfies the following properties
\begin{equation*}
 |K(x,y)|\leq \dfrac{C}{|x-y|} \qquad |\del_x K(x,y)| \leq \dfrac{C}{|x-y|^2},
\end{equation*}
for all $(x, y) \in \mathbb{R}^+ \times w$, $x \neq y$.
\end{lem} 

Based on the estimates furnished on the kernel $E$, we have the following lemma.
\begin{lem}\label{lem-Ch}
The operator $\mathbb{E}$ maps from $L^\infty(w\times (0,\delta))$ to ${\rm LogLip}(\R^2_+)$. For all relatively compact open set $w_1\subset w$, all $\delta^\prime<\delta$ and for all $\chi \in C_c^\infty(w\times [0,\delta))$ such that $\chi=1$ on a neighborhood of $\bar{w}_1\times [0,\delta'],$ there is a constant $C$ such that for all $g\in L^\infty$ supported in $w_1 \times (0,\delta^\prime)$ there holds
\begin{equation}\label{Eg-est}
\| \chi \mathbb{E} g \|_{{\rm LogLip}(\R^2_+)}+\| x_n \chi\del_j \mathbb{E} g \|_{{\rm LogLip}(\R^2_+)} \leq C \| g \|_{L^\infty(\R^2_+)}.
\end{equation}
\end{lem}
Due to Lemma~\ref{lem-E}, Lemma~\ref{lem-Ch} becomes as a consequence of Proposition~\ref{prop-MarPul} stated above.

Now, we move to prove Theorem~\ref{theo-BM-log}.
\begin{proof}[Proof of Theorem~\ref{theo-BM-log}]
 We take $\chi$ and $\bar{\chi}$ in $C^\infty_c(w\times [0,\delta))$ such that $\chi=1$ on a neighborhood of the support of $\phi$ and $\bar{\chi}=1$ of the support of $\chi$. By assumption $\phi$ and $g=\mathcal{L}\phi$ have compact support in $w\times [0,\delta).$ Let $\Psi=\chi \mathbb{E} g$. We compute
\[
\mathcal{L}\Psi=\chi \mathcal{L}\mathbb{E}g +[\mathcal{L}, \chi]\mathbb{E}g=g+\chi \mathbb{K}g+[\mathcal{L}, \chi]\bar{\chi}\mathbb{E}g.
\]
This implies
\begin{equation*}
\mathcal{L}(\Psi-\phi)=\chi \mathbb{K}g+[\mathcal{L}, \chi]\bar{\chi}\mathbb{E}g.
\end{equation*}
{\sc Bresch} and {\sc M\'etivier} claimed that the right-hand side term in the above equation, say, $h$, is in $C^\beta(\overline{\R^2_+})$. This is true since the operator $\mathbb{E}$ maps from $L^p(\omega \times (0,\delta))$ to $W^{1,p}_{1,\rm{loc}}(\overline{\R^2_+})$, and the operator $\mathbb{K}$ with kernel $K$ acts from $L^p(\omega \times (0,\delta))$ to $C^\beta(\overline{\mathbb{R}}^2_+)$ for all $\beta\leq 1-2/p.$ Accordingly, the following estimates hold
\begin{equation}\label{esti-h}
\begin{split}
& \| \Psi \|_{C^\beta_1(\overline{\mathbb{R}^2_+})}\leq C\| \mathbb{E}g \|_{W^{1,p}_{1,\rm{loc}}(\overline{\R^2_+})} \leq C \|g \|_{L^p(\omega \times (0,\delta))},\\
&\| [\mathcal{L}, \chi]\bar{\chi}\mathbb{E}g\|_{ C^{\beta}(\overline{\R^2_+})}\leq C \|g \|_{L^p(\omega\times(0,\delta))},\\
&\|\chi \mathbb{K}g\|_{ C^{\beta}(\overline{\R^2_+})}\leq C \|g \|_{L^p(\omega\times(0,\delta))},
\end{split}
\end{equation}
for any $\beta\leq 1-2/p$. On the other hand, one could apply the result presented in \cite[Theo. 1]{gou-shima} to deduce that if $h\in C^\beta(\overline{\R^2_+}),$ then we have
\begin{equation*}
\Psi-\phi \in C^{\beta+1}(\overline{\R^2_+})\qquad x_n(\Psi-\phi) \in C^{\beta+2}(\overline{\R^2_+}),
\end{equation*}
and the following estimate holds
\[
\| \Psi-\phi \|_{{\rm LogLip}(\R^2_+)}+\| x_n \del_j (\Psi-\phi) \|_{{\rm LogLip}(\R^2_+)} \leq C \big( \| h\|_{C^{\beta}(\overline{\R^2_+})} +\| \Psi-\phi \|_{C^{\beta}_1(\overline{\R^2_+})}\big).
\]
This estimate coupled with the log-Lipschitz estimate on $\Psi=\chi \mathbb{E}g$ \eqref{Eg-est}, as well with \eqref{esti-h}, finish the proof of Theorem~\ref{theo-BM-log}. 
\end{proof}

\section{Classical solutions for the inviscid lake equations: proof of Theorem~\ref{Th-WP}}\label{sec-pointfixed}

The main idea is to introduce the characteristic curve along the flow and to use an iteration procedure based on the well-posedness of the linear transport equation. We refer to {\sc Marchioro} and {\sc Pulvirenti} \cite{MarPul} for this type of construction and to \cite{ADL} (proof of Theorem 2.2 in Section 7.1) where all the details are included in the context of $C^1-$solutions.

Without any loss of generality, we will assume that ${\rm diam} (\Omega)<1$ which will simplify the use of the log-Lipschitz estimate: 
\[
1+|\ln |x-y||\leq -C_{\Omega} \ln|x-y|\quad \forall x,y\in \Omega.
\]

We fix $(\Omega,b)$ verifying \eqref{b-cond} and an initial data $\omega_{0}\in C^1_{c} (\Omega)$. We denote $C_{\rm lip}:=C\|\omega_{0}\|_{L^\infty}$ where $C$ is the constant appearing in \eqref{log-lip} of Theorem~\ref{ellip-est}, and we define
\[
\Omega_{T} = \Big\{x\in \Omega\ |\ \dist(x,\partial\Omega) \geq \dist(\supp \omega_{0},\partial\Omega)^{\exp(C_{\rm lip}T)} \Big\}.
\]
For any $T >0$, we introduce now the subspace $C_{\omega_0,T}\subset C^1([0,T]\times \Omega)$ as follows: a function $\omega \in C^1([0,T]\times \Omega)$ belongs to $C_{\omega_0,T}$ if and only if:
\begin{itemize}
\item $\|\omega\|_{L^\infty(\Omega)}=\|\omega_0\|_{L^\infty(\Omega)}$, for every $t\in [0,T]$,
\item $\omega(0,x)=\omega_0(x)$, for every $x\in \Omega$,
\item $\supp \omega(t,\cdot) \subset \Omega_{T}$, for every $t\in [0,T]$.
\end{itemize}
Of course, the subspace $C_{\omega_0,T}$ inherits its topology from the metric of $C^1\left([0,T]\times\Omega\right)$.

We define the trajectories $X$ starting from any $x\in \supp \omega_{0}$: the curve solving the differential equation
\[
\frac{d X(t,x)}{d t}=u(t,X(t,x)), \quad X(0,x)=x.
\]

In the sequel of the section, we will use several times the elliptic estimates furnished in Theorem~\ref{ellip-est} to get estimates for the velocity $u=\frac{1}{b}\nabla^\perp \psi$, associated to $\omega$ through the elliptic problem \eqref{ellip-eq} with $f=b\omega$. These estimates are collected in the following lemma.

\begin{lem}\label{lem-iteration}
There exists a constant $C_{0,T}$ which depends only on $\omega_{0}$, $T$ and on the geometry of the lake $(\Omega,b)$ such that for any $\omega\in C_{\omega_0,T}$, we have:
\begin{itemize}
\item the velocity $u=\frac{1}{b}\nabla^\perp \psi$ associated to $\omega$ through the elliptic problem \eqref{ellip-eq} with $f=b\omega$ satisfies 
\[
\| u(t,\cdot)\|_{L^{\infty}(\Omega)} \leq C_{0,T}
\]
and 
\[\| \nabla u(t,\cdot)\|_{L^{\infty}(\Omega_{T})}\leq C_{0,T} \ln \Big(2+\|\nabla \omega(t,\cdot)\|_{L^\infty(\Omega)}\Big)\quad \text{for all $t\in [0,T]$};
\]
\item the unique characteristic curve $X(\cdot,x)\in C^1([0,T] ; \Omega)$ associated to $u$ defines for any $t\in [0,T]$ a $C^1$ diffeomorphism from $\supp \omega_{0}$ onto its image;
\item the function $\tilde \omega$ defined as
\begin{equation*}
\left\{
\begin{array}{ll}
 \tilde\omega(t,x) = \omega_{0}( X(t,\cdot)^{-1}(x)) , & \text{if } x\in X(t,\supp \omega_{0}),\\
 \tilde\omega(t,x) =0,& \text{otherwise,} 
\end{array}\right.
\end{equation*}
belongs to $C_{\omega_{0},T}$.
\end{itemize}
\end{lem}

\begin{proof} For $\omega\in C_{\omega_0,T}$, we consider $u$ the unique solution of \eqref{ellip-eq} with $f=b\omega$. By Theorem~\ref{ellip-est}, we conclude the $C^1-$estimate of $u$ and we define uniquely the trajectory from any $x\in \supp \omega_{0}$: $X(\cdot,x)\in C^1([0,T_{x}] ; \Omega)$ where $T_{x}\leq T$ is such that $X(\cdot,x)\in \Omega$ for any $t\in [0,T_{x}]$. We can choose $T_{x}=T$ unless $X(\cdot,x)$ reaches the boundary in finite time. By the log-Lipschitz and the tangency boundary condition, this collision could never happen: indeed, the velocity being log-Lipschitz up to the boundary and tangent to the boundary, we can also define the trajectories from $y\in \partial\Omega$: $X(\cdot,y)\in C^1([0,T] ; \partial\Omega)$ and we compute
\begin{align*}
 \frac{d}{dt}-\ln |X(t,x)-X(t,y)| &= -\frac{(X(t,x)-X(t,y))\cdot (u(t,X(t,x))-u(t,X(t,y)))}{ |X(t,x)-X(t,y)|^2}\\
 &\leq C_{\rm lip}(-|\ln |X(t,x)-X(t,y)| |),
\end{align*}
hence by Gr\"onwall's lemma we have 
\[
\dist(X(t,x),\partial\Omega) \geq \dist(x,\partial\Omega)^{\exp(C_{\rm lip}T)}, \forall x\in [0,T_{x}].
\]
This allows us to take $T_{x}=T$. By the regularity of $u$, we state that the characteristics define a $C^1$ diffeomorphism which allows to define $\tilde \omega$. This function belongs obviously to $C_{\omega_0,T}$ by definition.
\end{proof}

This lemma allows us to perform the usual iteration procedure scheme, that we present in the sequel of this section.

\subsection{Construction of an approximating sequence}

First, we build an approximating sequence $(\omega_{n} )_{n\in \N}$ using a standard iteration procedure based on the well-posedness of the linear transport equation.

The first term is simply given by the constant-in-time function $\omega_0(t,x)=\omega_0(x)$, for all $(t,x)\in [0,T]\times \Omega$. Then, for each $\omega_n\in C_{\omega_0,T}$, the following term $\omega_{n+1}$ is defined as the unique solution to the linear transport equation
\[
\left\{
\begin{aligned}
& \partial_{t} \omega_{n+1} + u_n\cdot \nabla \omega_{n+1} =0,\\
& \omega_{n+1}(0,\cdot) = \omega_0,
\end{aligned}
\right.
\]
where the velocity flow $u_{n}$ is given as the unique solution of
\[
\curl u_{n}=b\omega_{n} \text{ in }\Omega,\quad \divv (b u_{n})=0 \text{ in }\Omega, \quad (bu_{n})\cdot n=0 \text{ on }\partial\Omega.
\]
By Lemma~\ref{lem-iteration}, $\omega_{n+1}$ is well defined though the characteristics $X_{n}$ associated to $u_{n}$ and belongs to $C_{\omega_{0},T}$.

\subsection{Uniform boundedness in $C^1$}

Second, we establish uniform $C^1$-bound on this approximating sequence. To this end, we observe that the $\omega_n$'s also solve (in the sense of distributions) the following equation, for $i=1,2$:
\[
\partial_{t} \partial_{x_i}\omega_{n+1} + u_n\cdot \nabla \partial_{x_i}\omega_{n+1} =-\partial_{x_i}u_n\cdot \nabla \omega_{n+1}.
\]
It follows that, for any $[a,b]\subset [0,T]$,
\[
\partial_{x_i}\omega_{n+1}(b,X_n(b,x)) = \partial_{x_i}\omega_{n+1}(a,X_n(a,x)) -\int_a^b \partial_{x_i}u_n\cdot \nabla \omega_{n+1}(s,X_n(s,x))ds,
\]
hence, we obtain by Gr\"onwall's lemma
\begin{align*}
|\nabla\omega_{n+1}(b,X_n(b,x))| & \leq |\nabla\omega_{n+1}(a,X_n(a,x))| e^{\int_a^b |\nabla u_n(s,X_n(s,x))| ds},\\
|\partial_t\omega_{n+1}(b,X_n(b,x))| & \leq |\nabla\omega_{n+1}(a,X_n(a,x))| |u_n(b,X_n(b,x))| e^{\int_a^b |\nabla u_n(s,X_n(s,x))| ds},
\end{align*}
for all $x\in\supp \omega_0$. By the $C^1$ estimate of $u_{n}$ included in Lemma~\ref{lem-iteration}, we conclude that 
\begin{align*}
 \|\omega_{n+1}\|_{C^1([a,b]\times \Omega)} 
\leq& \| \omega_{0} \|_{ L^\infty(\Omega)} \\
&+C_{0,T} \|\nabla\omega_{n+1}(a,\cdot)\|_{L^\infty(\Omega)}
 e^{(b-a)C_{0,T} \ln (2 + \| \nabla \omega_{n} \|_{L^\infty ([a,b]\times\Omega)})}.
\end{align*}
Hence
\[
\|\omega_{n+1}\|_{C^1([a,b]\times\Omega)}\leq C_{0,T}+C_{0,T}\|\nabla\omega_{n+1}(a,\cdot)\|_{L^\infty(\Omega)} (2+\|\omega_n\|_{C^1([a,b]\times\Omega)})^{C_{0,T}(b-a)}, 
\]
where we recall that $C_{0,T}>0$ may only depend on $\|\omega_0\|_{L^\infty(\Omega)}$, $T$, but is independent of $\omega_n$, $\omega_{n+1}$ and $[a,b]$. Setting $(b-a)$ sufficiently small, for instance,
\begin{equation}\label{small time 2}
C_{0,T}(b-a)\leq \frac 12,
\end{equation}
yields to
\[
\|\omega_{n+1}\|_{C^1([a,b]\times\Omega)}\leq 1 +C_{0,T}+ C_{0,T}^2 \|\nabla\omega_{n+1}(a,\cdot)\|_{L^\infty(\Omega)}^2 +\frac 12\|\omega_n\|_{C^1([a,b]\times\Omega)},
\]
whence, for each $k=0,\ldots,n$,
 \begin{align*}
 \|\omega_{n+1}\|_{C^1([a,b]\times\Omega)} 
 \leq & \Big(1+C_{0,T}+C_{0,T}^2 \sup_{p\geq 0}\|\nabla\omega_{p}(a,\cdot)\|_{L^\infty(\Omega)}^2 \Big)\Big(\sum_{j=0}^k2^{-j}\Big) \\
 &+\frac 1{2^{k+1}}\|\omega_{n-k}\|_{C^1([a,b]\times\Omega)} \\
 \leq & 2+2C_{0,T}+ 2C_{0,T}^2\sup_{p\geq 0}\|\nabla\omega_{p}(a,\cdot)\|_{L^\infty(\Omega)}^2 +\frac 1{2^{n+1}}\|\omega_{0}\|_{C^1(\Omega)}.
 \end{align*}
 
Since the initial data $\omega_0$ belongs to $C^1(\Omega)$, the constant $C_{0,T}$ only depends on fixed parameters and the bound \eqref{small time 2} on the maximal length of $[a,b]$ only involves $C_{0,T}$, we deduce that we may propagate the preceding $C^1$-bound on $[a,b]$ to the whole interval $[0,T]$. This yields a uniform bound
\begin{equation}\label{C1 bound}
\sup_{n\geq 0} \|\omega_{n}\|_{C^1([0,T]\times\Omega)}<\infty.
\end{equation}

\subsection{Convergence properties}

Next, we show that $(\omega_{n} )_{n\in \N}$ is actually a Cauchy sequence in $C^0$, which allows us to pass to the limit in the iteration scheme and obtain a solution of \eqref{Inviscid-vort} in the sense of distributions.

To this end, note that
\begin{align*}
\partial_{t} (\omega_{n+1}-\omega_n) + u_n\cdot\nabla(\omega_{n+1}-\omega_n) & = (u_{n-1}-u_n)\cdot \nabla \omega_{n}, \\
\partial_{t} (\omega_{n+1}-\omega_n) + u_{n-1}\cdot\nabla(\omega_{n+1}-\omega_n) & = (u_{n-1}-u_n)\cdot \nabla \omega_{n+1},
\end{align*}
whence, for any $[a,b]\subset [0,T]$,
\begin{align*}
(\omega_{n+1}-\omega_n)(b,X_n(b,x))= & (\omega_{n+1}-\omega_n)(a,X_n(a,x))\\
&+ \int_a^b (u_{n-1}-u_n)\cdot \nabla \omega_{n}(s,X_n(s,x)) ds,\\
(\omega_{n+1}-\omega_n)(b,X_{n-1}(b,x))= &(\omega_{n+1}-\omega_n)(a,X_{n-1}(a,x))\\
& +\int_a^b(u_{n-1}-u_n)\cdot \nabla \omega_{n+1}(s,X_{n-1}(s,x))ds.
\end{align*}
By linearity of the elliptic problem, Theorem~\ref{ellip-est} states that
\[
\| u_{n-1}-u_{n} \|_{L^\infty(\Omega)}\leq C \| \omega_{n-1}-\omega_{n}\|_{L^\infty(\Omega)}.
\]
This implies, utilizing \eqref{C1 bound}, for each $k=0,\ldots, n-1$, that
\begin{equation}\label{cauchy estimate}
\begin{aligned}
\|\omega_{n+1}-\omega_n&\|_{L^\infty([a,b]\times\Omega)} 
\leq 
\|(\omega_{n+1}-\omega_n)(a,\cdot)\|_{L^\infty(\Omega)}\\
&\hspace{2.5cm}+
C_1(b-a)
\|\omega_{n}-\omega_{n-1}\|_{L^\infty([a,b]\times\Omega)}
\\
\leq &
\sum_{j=0}^k
(C_1(b-a))^j
\|(\omega_{n+1-j}-\omega_{n-j})(a,\cdot)\|_{L^\infty(\Omega)}
\\
& \qquad+
(C_1(b-a))^{k+1}
\|\omega_{n-k}-\omega_{n-1-k}\|_{L^\infty([a,b]\times\Omega)}
\\
\leq &
\sum_{j=0}^{n-1}
(C_1(b-a))^j
\|(\omega_{n+1-j}-\omega_{n-j})(a,\cdot)\|_{L^\infty(\Omega)}
\\
& \qquad+
(C_1(b-a))^{n}
\|\omega_{1}-\omega_{0}\|_{L^\infty([a,b]\times\Omega)},
\end{aligned}
\end{equation}
for some independent constant $C_1>0$. As before, we set $(b-a)$ sufficiently small, say,
\begin{equation*}
C_1(b-a)\leq \frac 12.
\end{equation*}
In particular, since the $\omega_n$'s all have the same initial data $\omega_0$, we find that
\begin{align*}
\|\omega_{n+1}-\omega_n\|_{L^\infty([0,b-a]\times\Omega)}
& \leq
(C_1(b-a))^{n}
\|\omega_{1}-\omega_{0}\|_{L^\infty([0,b-a]\times\Omega)}
\\
& \leq
\frac 1{2^{n-1}}
\|\omega_0\|_{L^\infty(\Omega)}.
\end{align*}
Therefore, utilizing the elementary identity
\begin{equation*}
\sum_{j=0}^n \begin{pmatrix}j+k\\ k\end{pmatrix}= \begin{pmatrix}n+k+1\\ k+1\end{pmatrix},
\end{equation*}
for each $n,k\in\mathbb{N}$, we obtain
\begin{align*}
\|\omega_{n+1}-\omega_n\|_{L^\infty([k(b-a),(k+1)(b-a)]\times\Omega)}
\leq &
2
\begin{pmatrix}n+k\\ k\end{pmatrix}
(C_1(b-a))^n
\|\omega_0\|_{L^\infty(\Omega)}
\\
\leq &
\frac 1{2^{n-1}}
\begin{pmatrix}n+k\\ k\end{pmatrix}
\|\omega_0\|_{L^\infty(\Omega)},
\end{align*}
whence
\begin{align*}
\|\omega_{n+1}-\omega_n\|_{L^\infty([0,T]\times\Omega)}
& \leq
\frac {C}{2^{\frac n2}},
&&\text{for all }n\geq 0,
\\
\|\omega_{m}-\omega_n\|_{L^\infty([0,T]\times\Omega)}
& \leq
\frac {C'}{2^{\frac n2}},
&&\text{for all }m>n\geq 0,
\end{align*}
for some independent constants $C,C'>0$.

It follows that $(\omega_n)_{n\geq 0}$ is a Cauchy sequence in $C^0([0,T]\times\Omega)$ and, therefore, there exists $\omega\in C([0,T]\times\Omega)$ such that
\begin{equation}\label{convergence xi mu}
\begin{aligned}
\omega_n & \longrightarrow \omega
\quad\text{in }L^\infty([0,T]\times\Omega),
\\
u_n & \longrightarrow u
\quad\text{in }L^\infty([0,T]\times \Omega),
\end{aligned}
\end{equation}
where $u$ is defined by $\frac{1}{b}\nabla^\perp \psi[\omega]$ and we have used Theorem~\ref{ellip-est} to derive the convergence of $u_n$ from that of $\omega_n$. It is then readily seen that $\omega$ solves \eqref{Inviscid-vort} in the sense of distributions.

\subsection{Regularity of solution and conclusion of proof}
In order to complete the proof of well-posedness in $C^1([0,T])$, there only remains to show that $\omega$ is actually of class $C^1$. Indeed, the uniqueness of solutions will be easily ensured from an estimate similar to \eqref{cauchy estimate}.

For the moment, the uniform boundedness of $(\omega_n)_{n\geq 0}$ in $C^1([0,T]\times\Omega)$ only allows us to deduce that $\omega$ is Lipschitz continuous (in $t$ and $x$). We also know from \eqref{convergence xi mu} that $\supp \omega(t,\cdot)\subset \Omega_{T}$ for any $t\in [0,T]$. Theorem~\ref{ellip-est} whence implies that $\nabla u$ exists and is continuous in $[0,T]\times\Omega_{T}$. It follows that the associated characteristic curve $X(t,x)$ solving
\[
\frac{dX }{ds}=u(s,X),
\]
for some given initial data $X(0,x)=x\in \supp \omega_{0}$, belongs to $C^1([0,T]\times \supp \omega_{0}; \Omega)$, see the proof of Lemma~\ref{lem-iteration} to state that $X(t,x)\in \Omega_{T}\Subset\Omega$. Since the mapping $x\mapsto X(t,x)$ is a $C^1$-diffeomorphism from $\supp\omega_0$ onto its own image, we consider its inverse $X^{-1}(t,x)$. By the uniqueness of the linear transport equation, we have that $\omega$ can be expressed though the characteristics:
\begin{equation*}
\left\{
\begin{aligned}
\omega(t,x) & = \omega_0(X^{-1}(t,x)), && \text{if } x\in X(t,\supp\omega_0),
\\
\omega(t,x) & = 0, && \text{otherwise},
\end{aligned}
\right.
\end{equation*}
which gives that $\omega \in C^1_c([0,T]\times\Omega)$. This ends the proof of Theorem~\ref{Th-WP}.

Following the last argument in the proof of Theorem 2.2 in Section 7.1 of \cite{ADL}, it is also possible to show that $X_n$ converges uniformly in $(t,x)\in [0,T]\times\supp\omega_0$ towards $X$. However, such a property is not necessary for Theorem~\ref{Th-WP}.

\section{The vanishing viscosity limit: proof of Theorem~\ref{Th-Vanishing}}\label{section-van-vis}

In this section, we fix $(\Omega,b)$ verifying \eqref{b-cond2}, and we denote by $C$ a numerical constant whose value may change from line to line, but never depending on $\mu$ and $u_{0}^\mu$.
\begin{proof}
The proof of Theorem~\ref{Th-Vanishing} is based on a classical energy method used, for instance, to prove some vanishing viscosity theorems for incompressible Navier-Stokes equations (\cite{kato, te-wa, padd, If-Pl, Gi-Ke}). Indeed, we denote by 
$$w^{\mu}=\umu-u.$$
Then the variable $w^\mu$ solves the following equation in the variational formulation (see Section~\ref{def-viscous})
\begin{multline}\label{Eq-w}
\del_{t}(bw^{\mu})+(b\umu\cdot \nabla)w^{\mu}+(bw^\mu\cdot\nabla)u-2\mu\divv(b \D(u^\mu))-2\mu \nabla(b\divv \umu) \\+b\nabla(p^{\mu}-q)=0,
\end{multline}
with the following free divergence equation $\divv(bw^{\mu})=0.$ Taking $w^{\mu}$ as a test function which is possible since $w^{\mu}$ has the right regularity\footnote{The Euler solution belongs to $V_{b}$, which is the relevant test functions space for the Navier boundary condition.}, we obtain the following equation
\begin{multline}\label{Es1-w}
 \dfrac{1}{2}\Dt \intO |w^{\mu}|^2\,b\,dx+2\mu \intO \D(\umu): \D(w^{\mu})\,b\,dx\\
 +2\mu \intO \divv \umu \, \divv w^{\mu}\,b\,dx
 +\mu \intBO \eta_\mu\umu \cdot w^{\mu}\,b\,ds =-\intO (w^{\mu}\cdot\nabla)u\cdot w^{\mu}\,b\,dx.
\end{multline}
For every $x, y\in \mathbb{R}^{2}$, if we denote by $z=x-y$, the following identity holds
\begin{equation*}
2(x\cdot z)=2\Big|z+\frac{y}{2}\Big|^{2}-\frac{1}{2}|y|^{2},
\end{equation*}
and the same goes for the scalar matrix product. So from \eqref{Es1-w}, we get
\begin{equation}\label{diff-w}
\begin{split}
& \dfrac{1}{2}\Dt \intO |w^{\mu}|^2\,b\,dx+2\mu \intO |\D\big(w^{\mu}+\frac{u}{2}\big)|^2\,b\,dx\\
&+2\mu \intO |\divv(w^{\mu}+\frac{u}{2}\big)|^2\,b\,dx +\mu \intBO \eta_\mu\big|w^{\mu}+\frac{u}{2} \big|^2\,b\,ds \\
 &\hspace{4cm}=\dfrac{\mu}{2}\intO |\D(u)|^2\,b\,dx +\dfrac{\mu}{2}\intO |\divv(u)|^2\,b\,dx
 \\
 &\hspace{4.4cm}-\intO (w^{\mu}\cdot\nabla)u\cdot w^{\mu}\,b\,dx+\mu\intBO \dfrac{\eta_\mu}{4} |u|^2\,b\,ds.
 \end{split}
\end{equation}
Using estimates furnished in Theorem~\ref{Th-WP} on the solution $u$, H\"older inequality, and the fact that $b$ is bounded from above, we get (remember $0\leq \eta_\mu\leq \eta \mu^{-\beta}$)
\begin{equation}\label{estimates}
\begin{split}
&\dfrac{\mu}{2}\intO |\D(u)|^2\,b\,dx
 +\dfrac{\mu}{2}\intO |\divv(u)|^2\,b\,dx
 \leq \mu C \|u\|_{H^1(\Omega)}^2,\\
&\mu\intBO \dfrac{\eta_\mu}{4} |u|^2\,b\,ds\leq \eta \mu^{1-\beta} C \|u\|_{H^1(\Omega)}^2,\\
&\Big|\int_{\Omega} (w^{\mu}\cdot\nabla)u\cdot w^{\mu}\,b\,dx\Big| \leq \|\nabla u\|_{L^\infty(\Omega)} \|w^{\mu}\|_{L^2_b(\Omega)}^2.
\end{split}
\end{equation}
Besides, from \eqref{diff-w} and taking in mind \eqref{estimates}, we deduce the following
\begin{equation*}
\dfrac{1}{2}\frac{d}{dt} \|w^{\mu}\|_{L^2_b(\Omega)}^2
\leq \|\nabla u\|_{L^\infty(\Omega)} \|w^{\mu}\|_{L^2_b(\Omega)}^2+\mu C \|u\|_{H^1(\Omega)}^2 +\eta \mu^{1-\beta} C \|u\|_{H^1(\Omega)}^2.
\end{equation*}
Let us note that for the vanishing topography on all the boundary, for instance for $\alpha>0$ when $(\Omega,b)$ verifies \eqref{b-cond2}, we can remove the last right-hand side term. However, as said in Section~\ref{def-viscous}, we keep this term for the sake of possible generality in the future where the weight $b$ could be taken not identically zero at one or more components of the boundary, for instance on an island.

After applying the Gr\"onwall's lemma, we obtain
\begin{multline*}
 \|w^{\mu}(t)\|_{L^2_b(\Omega)}^2 \leq\Big[ \|w^{\mu}(0)\|_{L^2_b(\Omega)}^2
 +(\mu +\eta\mu^{1-\beta}) C \int_0^t \|u\|_{H^1(\Omega)}^2\Big]\\
 \times \exp \Big(2\int_0^t \|\nabla u\|_{L^\infty(\Omega)} \,d\tau\Big).
\end{multline*}
As $C$ is independent of $\mu$, this ends the proof of Theorem~\ref{Th-Vanishing}.
\end{proof}

\begin{rmk}
 Notice that, when estimating the third right-hand side integral in \eqref{diff-w}, we used the fact that $\nabla u$ belongs to $L^\infty(\Omega).$ This estimate is furnished by Theorem~\ref{ellip-est} and was not shown in the previous works on the inviscid lake equations \cite{br-me-lake, la-pau}. Nevertheless, one could try to estimate this integral using a different way, precisely following Yudovich's method for the proof of uniqueness of weak solutions to the incompressible Euler equations. Indeed, following this way of proof, we just need to know that $\nabla u$ belongs to $L^p(\Omega),$ for all $\, p<\infty,$ which is due to \cite{br-me-lake}. However, another difficulty appears which is interpreted by the requirement of having $|\sqrt{b}w^{\mu}|^{\frac{2}{p}}\in L^{\infty}(\Omega)$. This information, unfortunately, is not known.

\end{rmk}

\medskip

\noindent
{\bf Acknowledgements.} The second author is partially supported by the Agence Nationale de la Recherche, project SINGFLOWS, grant ANR-18-CE40-0027-01. This work was realized during the secondment of the first author in Beijing International Center for Mathematical Research at Peking University.



\end{document}